\documentclass[10pt,a4paper]{amsart}
\usepackage{amsmath,amssymb,amsthm,wasysym,enumerate,url,tikz,xr}

\numberwithin{equation}{section}

\theoremstyle{plain}
\newtheorem{theorem}[equation]{Theorem}
\newtheorem{lemma}[equation]{Lemma}
\newtheorem{corollary}[equation]{Corollary}
\newtheorem{proposition}[equation]{Proposition}
\newtheorem{hyp}[equation]{Hypotheses}
\newtheorem{hypsing}[equation]{Hypothesis}

\newtheorem*{theorem*}{Theorem}
\newtheorem*{lemma*}{Lemma}
\newtheorem*{corollary*}{Corollary}
\newtheorem*{proposition*}{Proposition}
\newtheorem*{hyp*}{Hypotheses}
\newtheorem*{hypsing*}{Hypothesis}

\theoremstyle{definition}
\newtheorem{defn}[equation]{Definition}
\newtheorem*{defn*}{Definition}

\theoremstyle{remark}
\newtheorem{rem}[equation]{Remark}
\newtheorem*{rem*}{Remark}

\newtheorem*{rems*}{Remarks}

\newcommand{\N}{\ensuremath{\mathbb{N}}}

\newcommand{\calO}{\ensuremath{\mathcal{O}}}
\newcommand{\calL}{\ensuremath{\mathcal{L}}}
\newcommand{\calF}{\ensuremath{\mathcal{F}}}

\newcommand{\proj}[2]{\ensuremath{\mathbb{P}_{#1}^{#2}}}

\newcommand{\twomat}[4]{\begin{pmatrix} #1 & #2 \\ #3 & #4 \end{pmatrix}}

\makeatletter
\newenvironment{proofof}[1]{\par
  \pushQED{\qed}%
  \normalfont \topsep6\p@\@plus6\p@\relax
  \trivlist
  \item[\hskip\labelsep
        \bfseries
    Proof of #1\@addpunct{.}]\ignorespaces
}{%
  \popQED\endtrivlist\@endpefalse
}
\makeatother

\begin{document}
\title{Cocycle twists of 4-dimensional Sklyanin algebras}
\author{Andrew Davies}
\address{School of Mathematics\\ University of Manchester\\  Manchester\\  United Kingdom\\  M13 9PL}
\email{andrewpdavies@gmail.com}
\date{\today}

\begin{abstract}
We study cocycle twists of a 4-dimensional Sklyanin algebra $A$ and a factor ring $B$ which is a twisted homogeneous coordinate ring. Twisting such algebras by the Klein four-group $G$, we show that the twists $A^{G,\mu}$ and $B^{G,\mu}$ have very different geometric properties to their untwisted counterparts. For example, $A^{G,\mu}$ has only 20 point modules and infinitely many fat point modules of multiplicity 2. The ring $B^{G,\mu}$ falls under the purview of Artin and Stafford's classification of noncommutative curves, and we describe it using a sheaf of orders over an elliptic curve.
\end{abstract}

\maketitle

\section{Introduction}\label{sec: introduction}
One of the major open problems in noncommutative projective algebraic geometry is the classification of noncommutative \proj{}{3}'s, or their algebraic counterparts, noncommutative AS-regular algebras. Such algebras have been extensively studied in the previous 20 years and Stafford and Van den Bergh's survey paper \cite{stafford2001noncommutative} provides a good overview of the area. The main aim of this paper is to introduce a family of such algebras possessing significantly different properties from those known previously. The algebras we study are also related to Zhang's construction of twisted algebras \cite{zhang1998twisted}.

Among noncommutative AS-regular algebras some of the most interesting are the 4-dimensional Sklyanin algebras, studied in \cite{smith1992regularity, smith1994center} (see \S\ref{subsec: 4dimsklyanin} for a definition via generators and relations). Let us denote such an algebra by $A$. The structure of these algebras was described in \cite{smith1994center} via a factor algebra $B=B(E,\sigma,\mathcal{L})$, known as a twisted homogeneous coordinate ring over an elliptic curve $E$. Here $\sigma$ is an automorphism of $E$, while $\mathcal{L}$ is a very ample invertible sheaf. The curve $E$ parameterises the \emph{point modules} over $A$. Such a module is $\N$-graded, has Hilbert series $1/(1-t)$ and is generated in degree 0. Interestingly, much of the structure and many of the properties of $A$ are determined by the geometric triple $(E,\sigma,\mathcal{L})$ \cite{levasseur1993modules,smith1994four,stafford2001noncommutative}. When $\sigma$ has finite order the properties of $A$ have a different flavour; we will restrict our study to the case in which it has infinite order.

In the survey paper \cite{odesskii2002elliptic}, Odesski{\u\i} makes some brief remarks about a variant of $A$ which we label $A'$ for now. One can construct $A'$ by first defining an action of the Klein four-group $G$ on both $A$ and the matrix ring $M_2(k)$, where $k$ is the base field. One can then define
\begin{equation}\label{eq: odesskiiintro}
A':= (A \otimes_k M_2(k))^G.
\end{equation}
Thus $A'$ is an invariant ring under the action of $G$ that is induced on the tensor product of $A$ and $M_2(k)$.

The algebra $A'$ will be one of the major objects of study in this paper. We first show that $A'$ can be  described as a cocycle twist of $A$ in the sense of the companion paper \cite{davies2014cocycle1}. This description allows us to use results from that paper to conclude that $A'$ has several nice properties, as the following results demonstrates.
\begin{proposition}\label{prop: thmintro}
Let $A$ denote a 4-dimensional Sklyanin algebra associated to the elliptic curve $E$ and automorphism $\sigma$ which has finite order. Then $A'$ can be described as a cocycle twist of the form $A^{G,\mu}$ for some group $G$ and 2-cocycle $\mu$. Furthermore, it is a noetherian AS-regular domain of dimension 4.
\end{proposition}

We defer the definition of cocycle twists and 2-cocycles until \S\ref{subsec: cocycletwists}. Let us use the notation $A^{G,\mu}$ instead of $A'$ henceforth. 

The properties that $A^{G,\mu}$ possesses by Proposition \ref{prop: thmintro} are interesting in their own right, but perhaps more remarkable are some of its geometric properties, especially when contrasted with those of $A$. As the well as point modules---for which $A$ has a family parameterised by $E$---we also consider fat point modules, which are a natural generalisation. Results of Smith and Staniszkis in \cite{smith1993irreducible} show that $A$ possesses 4 isomorphism classes of fat point modules of multiplicity 2 when $\sigma$ has infinite order. One has the following result for the cocycle twist, comprising elements of Proposition \ref{prop: sklyaninfatpoints} and Theorem \ref{thm: finitepointscheme}.
\begin{theorem}
In contrast to $A$, $A^{G,\mu}$ has only 20 point modules. Furthermore, $A^{G,\mu}$ has a family of fat point modules of multiplicity 2 parameterised by $E^G$, the curve whose underlying structure comes from the orbit space of $E$ under a natural $G$-action.
\end{theorem}

To prove this result, we use the embedding of $A^{G,\mu}$ inside $M_2(A)$ from \eqref{eq: odesskiiintro}. This allows us to consider direct sums of point modules over $A$ as modules over $A^{G,\mu}$ upon restriction. Moreover, one can perform essentially the same cocycle twist on $A^{G,\mu}$ to recover $A$, which enables us to consider $A$ as a subalgebra of $M_2(A^{G,\mu})$. It is this duality that is the key to the proof.

Since the twisted homogeneous coordinate ring $B$ controls much of the geometry related to the Sklyanin algebra $A$, it is natural to ask if there is an analogue for $A^{G,\mu}$. The answer to this question is positive: one can restrict the twisting operation down from $A$ to the factor ring $B$ to produce a cocycle twist $B^{G,\mu}$. We prove the following result which shows that the geometry related to $B^{G,\mu}$ contrasts with that associated to $B$.
\begin{theorem}[{Theorem \ref{thm: bgmunoptmodules}}]
$B^{G,\mu}$ has no point modules.
\end{theorem}

One can in fact say more about the structure of the algebra $B^{G,\mu}$, for which we need to turn to Artin and Stafford's work in \cite{artin2000semiprime}. In simple terms, their work shows that connected graded noetherian semiprime algebras of GK dimension 2 can be described in terms of a \emph{twisted ring} over a curve. Such twisted rings are generalisations of twisted homogeneous coordinate rings. Our major result is the following theorem.
\begin{theorem}[{Theorem \ref{theorem: geomdescthcrtwist}}]
The ring $B^{G, \mu}$ can be described via Artin and Stafford's classification of noncommutative curves \cite{artin2000semiprime}. More precisely, it is the twisted ring of an order over the elliptic curve $E^G$.
\end{theorem}
This result is noteworthy in that it is one of the first non-trivial examples of an algebra from the classification in \cite{artin2000semiprime} to appear `in the wild'. This result enables us to show that the family of fat point modules of multiplicity 2 over $A^{G, \mu}$ that are parameterised by the curve $E^G$ restrict to $B^{G, \mu}$. Furthermore, in Proposition \ref{prop: fatpointsincohE} we are able to determine the geometric origin of such modules.

We remark that several other authors have previously studied AS-regular algebras of dimension 4 that have only finitely many point modules like $A^{G,\mu}$. Examples include the work of Vancliff and co-authors in \cite{vancliff1998some} and \cite{shelton1999some}. Nevertheless, the algebra $A^{G,\mu}$ seems to be markedly different to such examples, even those that also have 20 point modules. 

Let us end the introduction by considering recent work of Pym in \cite{pym2015quantum}. In that paper all families of Calabi-Yau deformations of a polynomial ring on four generators are described, with the generic element of each family given, the Sklyanin algebra being one such element. The algebra $A^{G,\mu}$ does not appear in any of these families, but should instead be thought of as a deformation of a particular super-commutative polynomial ring, where some of the generators anti-commute. It is interesting to pose the question of whether there exist analogues of Pym's result in this context, and whether the algebras described are cocycle twists of algebras in his original deformation classification. 

In a paper appearing on the arXiv in January 2015, Chirvasitu and Smith \cite{chirvasitu2015exotic} prove some of the same results. As noted in their paper, our results had already appeared on the internet in July 2014 \cite{davies2014thesis}.

\subsection{Contents}
Let us now give a description of the contents of this paper. In \S\ref{sec: background} we give the definition of 4-dimensional Sklyanin algebras and discuss their geometric properties in more detail. We also recall the twist constructions considered in \cite{davies2014cocycle1}. We then proceed to study the behaviour of modules under a cocycle twist in \S\ref{sec: modules}. 

Next we study a particular twist of a 4-dimensional Sklyanin algebra in \S\ref{sec: twist4dimsklyanin}. There are many similar twists to that which we study; we show that all such twists belong to the same family up to a change of parameters. After this we apply the machinery and results of \S\ref{sec: modules} to one such algebra.

In \S\ref{sec: thcr} we study a cocycle twist of a twisted homogeneous coordinate ring which is a factor ring of the 4-dimensional Sklyanin algebra. This twist will be described in terms of Artin and Stafford's classification of noncommutative curves in Theorem \ref{theorem: geomdescthcrtwist}.

\subsection{Notation}\label{subsec: notation}
Throughout, $k$ will denote an algebraically closed field and $G$ a finite abelian group. Furthermore, we assume that $\text{char}(k) \nmid |G|$.

The algebras we will consider in this paper will be associative, connected graded $k$-algebras. A $k$-algebra $A$ is called \emph{connected graded (c.g.)} if $A= \bigoplus_{n \in \N} A_n$ with $ab \in A_{n+m}$ for all $a \in A_n$ and $b \in A_m$, where additionally $A_0=k$ and $\text{dim}_k A_n < \infty$ for all $n \in \N$. The \emph{Hilbert series} of a c.g.\ algebra is the power series $H_A(t)=\sum_{n \in \N} (\text{dim}_k A_n)t^n$.

When describing relations in such an algebra, we will use shorthand notation for two kinds of commutator. For $x,y \in A$, define
\begin{equation*}
[x,y]:=xy-yx \text{ and } [x,y]_+ := xy+yx.
\end{equation*}

We will work with right modules unless otherwise stated. By $\text{Mod}(A)$ ($\text{mod}(A)$) we shall denote the category of (finitely generated) $A$-modules and by $\text{GrMod}(A)$ ($\text{grmod}(A)$) the category of (finitely generated) $\N$-graded $A$-modules. 

We use the letter $e$ to denote the identity element of a group $G$. The action of an element $g\in G$ on $a \in A$ by a $k$-algebra automorphism will be denoted by the superscript $a^g$. The group of group automorphisms of $G$ will be denoted by $\text{Aut}(G)$.

Although not needed explicitly, we give a definition of AS-regular algebras due to their fundamental importance in our work.
\begin{defn}[{\cite[\S 0]{artin1987graded}}]\label{defn: asregular}
Let $d \in \mathbb{N}$. A connected graded $k$-algebra $A$ is said to be \emph{AS-regular of dimension
$d$} if the following conditions are satisfied:
\begin{itemize}
 \item[(i)] $A$ has finite GK dimension;
 \item[(ii)] $\text{gldim }A=d$;
 \item[(iii)] $A$ is AS-Gorenstein. That is, when $k$ and $A$ are considered as $\N$-graded $A$-modules one has \begin{equation*}
\text{Ext}^i_A(k,A)= \left\{ \begin{array}{cl} k & \text{if }i=d, \\ 0 & \text{if }i \neq d. \end{array}\right.
\end{equation*}
\end{itemize}
\end{defn}

Further notation will be introduced as and when needed.

\section*{Acknowledgements}
The contents of this paper form part of the author's PhD thesis \cite{davies2014thesis}, completed under the supervision of Professor Toby Stafford. The author wishes to express their gratitude to Professor Stafford for his support and guidance throughout both the preparation of this paper and their PhD, as well as to the EPSRC for funding their study.

\section{Background}\label{sec: background}

\subsection{4-dimensional Sklyanin algebras}\label{subsec: 4dimsklyanin}

The algebras that will form the main focus of this paper are the \emph{4-dimensional Sklyanin algebras}. Such algebras were extensively studied in \cite{smith1992regularity,smith1993irreducible,levasseur1993modules}, not to mention \cite{tate1996homological}, where they are viewed as a special case of a more general construction. 

We will denote a 4-dimensional Sklyanin algebra by $A:=A(\alpha,\beta,\gamma)$, where $\alpha, \beta, \gamma \in k$ are some parameters. They are connected graded $k$-algebras---where $k$ is some field with $\text{char}(k)\neq 2$---generated by degree 1 elements $x_0,x_1,x_2$ and $x_3$ subject to the six quadratic relations given in \cite[Equation 0.2.2]{smith1992regularity}:
\begin{equation}\label{eq: 4sklyaninrelnsintro}
\begin{array}{ll}
[x_0,x_1]-\alpha[x_2,x_3]_+, & [x_0,x_1]_+ -[x_2,x_3], \\ \relax
[x_0,x_2]-\beta[x_3,x_1]_+, & [x_0,x_2]_+ -[x_3,x_1], \\ \relax
[x_0,x_3]-\gamma[x_1,x_2]_+, & [x_0,x_3]_+ -[x_1,x_2]. 
\end{array}
\end{equation}
The scalars must satisfy the conditions
\begin{equation}\label{eq: 4sklyanincoeffcond}
\alpha + \beta + \gamma + \alpha \beta \gamma = 0 \;\text{ and }\; \{\alpha,\beta, \gamma\} \cap \{0,\pm 1\}=\emptyset.
\end{equation}
The latter condition is not necessary, but we assume it to avoid some degenerate algebras which are either isomorphic to iterated Ore extensions or are not domains (see \cite[\S 1]{smith1992regularity}).

The construction of this algebra can also be phrased in terms of a smooth elliptic curve $E$ and the translation $\sigma$ associated to a point on the curve. By using this method one can generalise the definition of $A$ to define $n$-dimensional Sklyanin algebras for all integers $n \geq 3$, as in \cite{tate1996homological}. 

There exist two central elements $\Omega_1,\Omega_2 \in A_2$ by \cite[Corollary 3.9]{smith1992regularity}. In fact, when $\sigma$ has infinite order one has $Z(A)=k[\Omega_1,\Omega_2]$ by \cite[Proposition 6.12]{levasseur1993modules}. The factor ring $A/(\Omega_1,\Omega_2)$ is isomorphic to the twisted homogeneous coordinate ring $B:=B(E,\calL,\sigma)$, where $\calL$ is a very ample invertible sheaf that provides an embedding $E \hookrightarrow\proj{k}{3}$.

In order to link the geometry described above with algebraic properties we need the following definition. 
\begin{defn}\label{def: noncommspace}
Let $\text{fdmod}(A)$ denote the full subcategory of $\text{grmod}(A)$ consisting of finite-dimensional modules. Define the \emph{category of tails of $A$} to be the quotient category
\begin{equation*}
\text{qgr}(A):= \text{grmod}(A)/\text{fdmod}(A).
\end{equation*}
The \emph{tail} of a module $M \in \text{grmod}(A)$ refers to the image of $M$ under the canonical functor $\pi: \text{grmod}(A)\rightarrow \text{qgr}(A)$.
\end{defn}

The irreducible objects in $\text{qgr}(A)$ are the tails of modules which are 1-critical with respect to GK dimension. Such an object is the tail of a module $M \in \text{grmod}(A)$ for which $\text{GKdim}(M) = 1$ and all proper graded submodules are finite-dimensional. 

By \cite[Proposition 7.1]{smith1994four} all irreducible objects in $\text{qgr}(A)$ arise as tails of the following class of modules. 
\begin{defn}[{\cite[pg. 8]{artin1990geometry}}]\label{def: fatpointmodule}
A \emph{fat point module} of $A$ is a module that satisfies the following conditions:
\begin{itemize}
 \item[(i)] it has Hilbert series $e/(1-t)$ for some integer $e \geq 1$;
 \item[(ii)] it is generated in degree 0;
 \item[(iii)] any non-zero $\N$-graded submodule has finite codimension.
\end{itemize}
The integer $e$ is called the \emph{multiplicity} of the module. Fat point modules of multiplicity 1 are called \emph{point modules}.
\end{defn}

The point modules over $A$ are parameterised by the elliptic curve $E$ and four extra points \cite[Proposition 2.4]{smith1992regularity}. Those on the curve $E$ are also point modules over $B$, and in fact these constitute all of its 1-critical modules (a consequence of the Noncommutative Serre's Theorem \cite[Theorem 1.3]{artin1990twisted}). When $\sigma$ has infinite order, which is the case of interest for us, $A$ has four fat point modules of each multiplicity greater than 1 up to isomorphism \cite[Main Theorem]{smith1993irreducible}. The fat point modules of higher multiplicity are linked to the 2-torsion points on $E$.

\subsection{Cocycle Twists}\label{subsec: cocycletwists}
We now briefly recall the two constructions of cocycle twists from \cite[\S 3]{davies2014cocycle1}. A result of Bazlov and Berenstein \cite[Lemma 3.6]{bazlov2012cocycle} shows that the two constructions produce isomorphic algebras. Both constructions require a choice of isomorphism $G \cong G^{\vee}$, where $g \mapsto \chi_g$. We fix such an isomorphism throughout \S\ref{subsec: cocycletwists}.

The first construction uses the fact that an action of $G$ by $\N$-graded automorphisms on $A$ induces an $(\N,G)$-bigrading. Supposing that $G$ acts on $A$ by such an automorphism, we can define a $G$-grading by setting $A_g = A^{\chi_{g^{-1}}}$ for all $g \in G$. This is the isotypic component of $A$ corresponding to the character $\chi_{g^{-1}}$, thus $a \in A_g$ if and only $a^h = \chi_{g^{-1}}(h)a$ for all $h \in G$. One can then use a 2-cocycle $\mu$ to twist the multiplication in $A$; this is a function $\mu: G \times G
\rightarrow k^{\times}$ satisfying the following relations for all $g,h,l \in G$:
\begin{equation*}
\mu(g,h)\mu(gh,l)=\mu(g,hl)\mu(h,l), \;\;\; \mu(e,g)=\mu(g,e)=1.
\end{equation*}

A twisted multiplication $\ast_{\mu}$ can be defined for homogeneous elements $a \in A_g$ and $b \in A_h$ by $a \ast_{\mu} b:= \mu(g,h)ab$, with juxtaposition denoting the original multiplication in $A$. We denote the algebra obtained by using this twisted multiplication by $A^{G,\mu}$. 

The second construction generalises an example of Odesski{\u\i} in \cite{odesskii2002elliptic}. Let $\mu$ be a 2-cocycle as above, and consider the twisted group algebra $kG_{\mu}$. Define an action of $G$ on $kG_{\mu}$ by $g^{h}:=\chi_g(h)g$ for all $g, h \in G$ and extending $k$-linearly. Given this action we can consider the action of $G$ on $AG_{\mu}$ defined by $(ag)^h:=a^h g^h= \chi_{g}(h)a^h g$ for all $a \in A$ and $g, h \in G$. The algebra in which we are interested is the invariant ring under this action, $(AG_{\mu})^G$.

\section{Modules under twisting}\label{sec: modules}
In this section we explore the interplay between modules over $A$ and those over a cocycle twist $A^{G,\mu}$. Here $G$ is fixed to be the Klein four-group and $\mu$ is a 2-cocycle of $G$ defined in \eqref{eq: 2cocycle}. More precisely we prove that, for algebras like the Sklyanin algebra $A$, the fat point modules over the algebra and its twist are closely related; one can construct fat point modules of multiplicity 2 over $A^{G,\mu}$ by taking direct sums of point modules over $A$.

Let $M$ be an $\N$-graded module over $A$ and $g \in G$. We can define a new $A$-module $M^g$ via the multiplication $m \ast_g a := ma^g$ for all $m \in M$ and $a \in A$. The underlying $\N$-graded vector space structure of $M$ remains unchanged, from which it follows that point modules are preserved under this process. 

We will consider such twisted modules in the first result of this section, but beforehand let us introduce some further notation. Let $I$ be a right ideal in $A$. Then for all $g \in G$ define $g(I) := \{a^{g} : \; a \in I\}$. 
\begin{lemma}\label{lem: annGinvariant}
Let $M$ be a $\N$-graded module over $A$, on which a finite group $G$ acts by $\N$-graded algebra automorphisms. Then
$\text{Ann}_A(M^g)=g^{-1}(\text{Ann}_A(M))$, and the ideal $\bigcap_{g \in G} \text{Ann}_A(M^g)$ is $G$-invariant. 

In particular, the ideal consisting of elements which annihilate all point modules
over $A$ is $G$-invariant.
\end{lemma}
\begin{proof}
Let $a \in \text{Ann}_A(M^g)$. Then for all $m \in M$, $ma^g=0$, hence $a^g \in \text{Ann}_A(M)$. This implies that $a
\in g^{-1}(\text{Ann}_A(M))$. The other inclusion is proved by the reverse argument. By the first part of the result 
\begin{equation}\label{eq: annptmod}
\bigcap_{g \in G} \text{Ann}_A(M^g) = \bigcap_{g \in G} g^{-1}(\text{Ann}_A(M)),
\end{equation}
hence it is preserved under the action of $G$. 

As discussed prior to the lemma, if $M$ is a point module over $A$ then $M^g$ is also a point module for all $g \in G$. The result is now clear.
\end{proof}

Let us now fix the hypotheses that we will work under for the remainder of this section.
\begin{hyp}\label{hyp: genhypforfatpts}
Let $k$ be an algebraically closed with $\text{char}(k)\neq 2$. Assume that $A$ is a $k$-algebra that satisfies the hypotheses of \cite[Theorem 1.4]{shelton1999embedding}, that is:
\begin{itemize}
 \item[(i)] it is generated in degree 1 and has Hilbert series $1/(1-t)^4$;
 \item[(ii)] it is noetherian and Auslander regular of global dimension 4;
 \item[(iii)] it is Cohen-Macaulay.
\end{itemize}
We denote the degree 1 generators of $A$ by $x_0, x_1 ,x_2$ and $x_3$. 

Let $G=\langle g_1,g_2 \rangle$ be the Klein four-group, which acts on $A$ by $\N$-graded algebra automorphisms. Furthermore, assume that the action of $G$ on $A_1$ affords the regular representation, inducing the following $G$-grading on generators:
\begin{equation}\label{eq: gengrading}
x_0 \in A_{e},\; x_1 \in A_{g_{1}},\; x_2 \in A_{g_{2}},\; x_3 \in A_{g_{1}g_{2}}.
\end{equation}
Finally, assume that $\mu$ is the 2-cocycle of $G$ defined by
\begin{equation}\label{eq: 2cocycle}
\mu(g_1^p g_2^q, g_1^r g_2^s) = (-1)^{ps}
\end{equation}
for all $p, q, r, s \in \{ 0 , 1\}$. 
\end{hyp}

Note that $kG_{\mu}\cong M_2(k)$ under these hypotheses, via the map sending
\begin{align}\label{eq: diagmat}
e \mapsto \begin{pmatrix} 1 & 0 \\ 0 & 1 \end{pmatrix},\; g_1 \mapsto \begin{pmatrix} 1 & 0 \\ 0 &
-1 \end{pmatrix},\; g_2 \mapsto \begin{pmatrix} 0 & 1 \\ 1 & 0 \end{pmatrix}, \;\;g_1g_2 \mapsto \begin{pmatrix} 0 & -1 \\ 1 & 0 \end{pmatrix}.
\end{align}

Since $A^{G,\mu} \subset (AG_{\mu})^G$, the cocycle twist embeds inside $M_2(A)$. This embedding is defined in the following lemma, whose proof follows from \eqref{eq: diagmat} and \cite[Lemma 3.6]{bazlov2012cocycle}.
\begin{lemma}\label{lem: matrixgens}
The degree 1 generators of $A^{G,\mu}$, denoted by $v_0, v_1, v_2$ and $v_3$, are given by the following matrices in $M_2(A)$:
\begin{align}\label{eq: matrixembedding}
v_0 = \begin{pmatrix} x_0 & 0 \\ 0 & x_0 \end{pmatrix},\; v_1 = \begin{pmatrix} x_1 & 0 \\ 0 &
-x_1 \end{pmatrix},\; v_2 = \begin{pmatrix} 0 & x_2 \\ x_2 & 0 \end{pmatrix}, \;\;
v_3 = \begin{pmatrix} 0 & -x_3 \\ x_3 & 0 \end{pmatrix}.
\end{align}
\end{lemma}

Let us now introduce some notation for point modules over $A$. Under Hypotheses \ref{hyp: genhypforfatpts} one may apply \cite[Theorem 1.4]{shelton1999embedding} to conclude that point modules over $A$ are parameterised by its point scheme $\Gamma \subseteq \proj{k}{3}$, with the shifting operation on such modules controlled by a scheme automorphism $\sigma$. We will denote the point module corresponding to a point $p=(p_0,p_1,p_2,p_3) \in \Gamma$ by $M_p:= \bigoplus_{j \in \N} km_j^p$. For the action of the generators of $A$ on $M_p$ we will use the notation $m_j^p \cdot x_i:= \alpha_{j,i}^p m_{j+1}^p$, where $\alpha_{j,i}^p \in k$. By standard point module theory one has $M_p[j]_{\geq 0} \cong M_{\sigma^j(p)}$ for all $j \in \N$. 

For a point $p \in \Gamma$ it is clear that $M_p^2$ is an $\N$-graded right $M_2(A)$-module with Hilbert series $2/(1-t)$. In some cases such a module is actually a fat point module over the cocycle twist, as the following result shows.
\begin{proposition}\label{prop: fatpoints} 
Suppose that at least three coordinates of $p \in \Gamma$ are non-zero. Then $M_p^2$ is a fat point module over $A^{G,\mu}$ of multiplicity 2.
\end{proposition}
\begin{proof}
Let $N$ denote the submodule generated by $(M_p^2)_0$. We prove by induction that $(M_p^2)_j \subseteq N$ for all $j \geq 0$, where the base case $j=0$ is clear. Suppose that $(M_p^2)_j \subseteq N$ for some $j \geq 0$. Since $M_p[j]_{\geq 0} \cong M_{\sigma^j(p)}$, the action of the generators of $A$ on $m_j^p$ is given by the coordinates of $\sigma^j(p)$. At least
one generator does not annihilate $m_j^p$, in which case letting the corresponding generator of $A^{G,\mu}$ act on $(m_j^p,0)$ and $(0,m_j^p)$ shows that $(M_p^2)_{j+1} \subseteq N$. By induction, $N=M_p^2$ and so $M_p^2$ is generated in degree 0.

To prove that $M_p^2$ is 1-critical it is sufficient to show that any cyclic $\N$-graded submodule has finite codimension. Consider the submodule generated by an element $(m_j^p,\lambda m_j^p) \in M_p^2$, where $\lambda \in k^{\times}$. By assumption $\alpha_{j,i}^p \neq 0$ for at least three of the generators. This means that either $\alpha_{j,0}^p,\alpha_{j,1}^p \neq 0$ or $\alpha_{j,2}^p,\alpha_{j,3}^p \neq 0$. In the former case $x_0$ and $x_1$ do
not annihilate $m_j^p$ and thus
\begin{align*}
(m_j^p,\lambda m_j^p) \cdot \left(v_0+ \frac{\alpha_{j,0}^p}{\alpha_{j,1}^p} v_1\right) &=
(2\alpha_{j,0}^p m_{j+1}^p,0), \\
(m_j^p, \lambda m_j^p) \cdot \left(v_0 - \frac{\alpha_{j,0}^p}{\alpha_{j,1}^p} v_1 \right) &=
(0,2 \lambda \alpha_{j,0}^p m_{j+1}^p).
\end{align*}
On the other hand, if $\alpha_{j,2},\alpha_{j,3} \neq 0$ then $x_2$ and $x_3$ do not annihilate $m_j^p$ and so
\begin{align*}
(m_j^p,\lambda m_j^p) \cdot \left(v_2 + \frac{\alpha_{j,2}^p}{\alpha_{j,3}^p} v_3 \right) &=
(2 \lambda \alpha_{j,3}^p m_{j+1}^p,0), \\
(m_j^p, \lambda m_j^p) \cdot \left(v_2 - \frac{\alpha_{j,2}^p}{\alpha_{j,3}^p} v_3 \right) &=
(0,2 \alpha_{j,2}^p m_{j+1}^p). 
\end{align*}

Thus $(m_{j+1}^p,0)$ and $(0,m_{j+1}^p)$ belong to the submodule generated by $(m_j^p, \lambda m_j^p)$, and therefore it has finite codimension in $M_p^2$ by the argument used in the induction earlier in the proof.

It remains to show that the submodules generated by either $(m_j^p,0)$ or $(0,m_j^p)$ have finite codimension. We give the argument for $(m_j^p,0)$, the argument for $(0,m_j^p)$ being similar. By assumption either $\alpha_{j,0}^p,\alpha_{j,2}^p \neq 0$ or $\alpha_{j,1}^p,\alpha_{j,3}^p \neq 0$. If $\alpha_{j,0}^p,\alpha_{j,2}^p \neq 0$ then one has
\begin{equation*}
(m_j^p,0)\cdot v_0 = (\alpha_{j,0}^p m_{j+1}^p,0)\; \text{ and }\; (m_j^p,0)\cdot v_2 = (0,\alpha_{j,2}^p m_{j+1}^p),
\end{equation*}
while if $\alpha_{j,1}^p,\alpha_{j,3}^p \neq 0$ one has
\begin{equation*}
(m_j^p,0)\cdot v_1 = (\alpha_{j,1}^p m_{j+1}^p,0)\; \text{ and }\;(m_j^p,0)\cdot v_3 = (0,-\alpha_{j,3}^p m_{j+1}^p).
\end{equation*} 
Once again, this is sufficient to show the submodule generated by $(m_j^p,0)$ has finite codimension.
\end{proof}

The action of $G$ on point modules allows us to define the following action of $G$ on $\Gamma$: for $p \in \Gamma$ and $g \in G$ define $p^g := q$ where $q \in \Gamma$ satisfies $(M_p)^g \cong M_q$.
\begin{lemma}\label{lem: actiononpoints}
Assume Hypotheses \ref{hyp: genhypforfatpts}. Then the group $G$ acts on a point $p=(p_0,p_1,p_2,p_3) \in \Gamma$ in the following manner:
\begin{gather}
\begin{aligned}\label{eq: Gactonpoints}
&p^e =p,\;\; p^{g_{1}}=(p_0,p_1,-p_2,-p_3),\;\; p^{g_{2}}=(p_0,-p_1,p_2,-p_3),\\ &p^{g_{1}g_{2}} =(p_0,-p_1,-p_2,p_3).
\end{aligned}
\end{gather}
In particular, this action preserves the condition on a point having at least three non-zero coordinates. 
\end{lemma}
\begin{proof}
Consider the point module $M_p=A/I_p$, where $I_p$ is a right ideal. For an element $g \in G$ we claim that $(M_p)^g
\cong A/g^{-1}(I_p)$. To see this, consider the map $\varphi: (M_p)^g \rightarrow A/g^{-1}(I_p)$ defined by $a + I_p \mapsto a^{g^{-1}}
+ g^{-1}(I_p)$ for $a \in A$. This map is well-defined and preserves the $\N$-graded vector space structures, thus it remains to check that it is an isomorphism of right $A$-modules. For $a, b \in A$ one
has
\begin{align*}
\varphi((a+I_p) \ast_g b) &= \varphi(ab^{g} + I_p) = (ab^g)^{g^{-1}} + g^{-1}(I_p) = a^{g^{-1}}b + g^{-1}(I_p),\\
\varphi(a+I_p)b &= (a^{g^{-1}} + g^{-1}(I_p))b = a^{g^{-1}}b + g^{-1}(I_p).
\end{align*}
Thus $\varphi$ is a homomorphism of $\N$-graded right $A$-modules. As $G$ acts by automorphisms, the modules $(M_p)^g$ and $A/g^{-1}(I_p)$ have the same Hilbert series. Consequently, $\varphi$ is an isomorphism.

Note that the right ideal $I_p$ is generated by the degree 1 elements $p_0x_j - p_jx_0$ for $j=1,2,3$. The isomorphism $(M_p)^g \cong A/g^{-1}(I_p)$ and \cite[Lemma 4.2]{davies2014cocycle1} indicate that the behaviour of the three
generators of $I_p$ under the action of $g^{-1}$ govern $p^g$. The result is clear when $g=e$ as its action is trivial. We give a proof for $g=g_1$, with the remaining two cases being similar. As $g_1$
has order 2, one has
\begin{gather}
\begin{aligned}\label{eq: ptmoddeg1gensact}
(p_0x_1 - p_1 x_0)^{g_{1}} &= p_0x_1^{g_{1}} - p_1 x_0^{g_{1}} = p_0x_1 - p_1 x_0,\\
(p_0x_2 - p_2 x_0)^{g_{1}} &= p_0x_2^{g_{1}} - p_2 x_0^{g_{1}} = -p_0x_2 - p_2 x_0,\\
(p_0x_3 - p_3 x_0)^{g_{1}} &= p_0x_3^{g_{1}} - p_1 x_3^{g_{1}} = -p_0x_3 - p_3 x_0.
\end{aligned} 
\end{gather}
The right ideal generated by the three elements on the right-hand side of \eqref{eq: ptmoddeg1gensact} corresponds to
the point $q = (p_0,p_1,-p_2,-p_3)$. Thus $p^{g_{1}}=q$ as in the statement of the lemma.
\end{proof}

In Proposition \ref{prop: fatpoints} we constructed fat point modules over the twist $A^{G,\mu}$. One can use such modules to construct direct sums of point modules over $A$ as we show in the next result. The key to being able to do this is performing the same cocycle twist on $A^{G,\mu}$ to recover $A$: to see this note that $A^{G,\mu}$ has the same underlying $G$-grading as $A$ and that $\mu$ takes only the values $\pm 1$. In particular, this gives us an embedding of $A$ inside $M_2(A^{G,\mu})$. 
\begin{proposition}\label{prop: fatpointsotherwaygen}
Consider $M_p^2$, a fat point module over $A^{G,\mu}$ constructed in Proposition \ref{prop: fatpoints}. The direct sum
$(M_p^2)^2$ is an $\N$-graded right $M_2(A^{G,\mu})$-module. On restriction to a module over the subalgebra $A$, one has
a decomposition of $\N$-graded right $A$-modules
\begin{equation}\label{eq: orbitdecomp}
(M_p^2)^2 \cong \bigoplus_{g \in G} M_{p^{g}}.
\end{equation}
\end{proposition}
\begin{proof}
The fat point module $M_p^2$ was obtained by restricting the $M_2(A)$-module $M_p^2$ to $A^{G,\mu}$. Thus
$(M_p^2)^2$ can be considered as the $M_4(A)$-module $M_p^4$, which becomes an $M_2(A^{G,\mu})$-module
upon restriction. One can then regard $M_p^4$ as an $A$-module by restricting a second time, with the action on homogeneous pieces given by $4 \times 4$ matrices. Explicitly, the embedding of $A$ into $M_4(A)$ is defined on generators by
\begin{align*}\label{eq: actonfatptdoubles}
x_0 &\mapsto \begin{pmatrix} x_0 &0&0&0 \\0& x_0 &0&0 \\ 0&0& x_0 &0 \\ 0&0&0& x_0
\end{pmatrix},\;\; 
x_1 \mapsto \begin{pmatrix} x_1 &0&0&0 \\0& -x_1 &0&0 \\ 0&0& -x_1 &0 \\ 0&0&0&
x_1 \end{pmatrix},\\ 
x_2 &\mapsto \begin{pmatrix} 0&0&0& x_2 \\0&0& x_2 &0 \\ 0& x_2 &0&0 \\ x_2 &0&0&0
\end{pmatrix},\;\; 
x_3 \mapsto \begin{pmatrix} 0&0&0& x_3 \\0&0& -x_3 &0 \\ 0& -x_3 &0&0 \\ x_3
&0&0&0 \end{pmatrix}.
\end{align*}

It is easy to see from these matrices that there is a decomposition of $A$-modules
\begin{equation}\label{eq: forbitdecomp}
M_p^4=(M_p,0,0,M_p) \oplus (0,M_p,M_p,0). 
\end{equation}

We claim that the decomposition in the statement of the proposition is then given by the isomorphisms of right
$A$-modules
\begin{gather}
\begin{aligned}\label{eq: decompasptmods}
(M_p,0,0,M_p) &\cong (m_0,0,0,m_0)A \oplus (m_0,0,0,-m_0)A \cong M_p \oplus M_{p^{g_{1}}}, \\
(0,M_p,M_p,0) &\cong (0,m_0,m_0,0)A \oplus (0,m_0,-m_0,0)A \cong M_{p^{g_{2}}} \oplus M_{p^{g_{1}g_{2}}}.
\end{aligned}
\end{gather}
To see this, note that the submodules 
\begin{equation*}
(m_0,0,0,m_0)A,\; (m_0,0,0,-m_0)A, \; (0,m_0,m_0,0)A \; \text{ and }\;(0,m_0,-m_0,0)A 
\end{equation*}
each have Hilbert series 1/(1-t) and are cyclic. By calculating which degree 1 elements of $A$ annihilate them, one can see that these cyclic submodules are isomorphic to the point modules indicated in \eqref{eq: decompasptmods}.
\end{proof}

Proposition \ref{prop: fatpointsotherwaygen} allows us to determine the isomorphisms in $\text{grmod}(A^{G,\mu})$
between the fat point modules constructed in Proposition \ref{prop: fatpoints}. 
\begin{corollary}\label{cor: fatpointisoclasses}
The only isomorphisms in $\text{grmod}(A^{G,\mu})$ between the fat point modules described in Proposition \ref{prop: fatpoints} are of the form $M_p^2 \cong M_{p^g}^2$ for all $g \in G$. 
\end{corollary}
\begin{proof}
To see that $M_p^2 \cong M_{p^g}^2$ holds for all $g \in G$, we first use the final part of Lemma \ref{lem:
actiononpoints}: if $p \in \Gamma$ has three non-zero coordinates then so does $p^g$ for all $g \in G$. Thus one can
construct $M_{p^g}^2$ as in Proposition \ref{prop: fatpoints} for any $g \in G$. The isomorphisms are governed by the
matrices in \eqref{eq: diagmat}:
\begin{itemize}
 \item[(i)] $M_p^2 \cong M_{p^{g_{1}}}^2$ via right multiplication by $\twomat{1}{0}{0}{-1}$;
 \item[(ii)] $M_p^2 \cong M_{p^{g_{2}}}^2$ via right multiplication by $\twomat{0}{1}{1}{0}$;
 \item[(iii)] $M_p^2 \cong M_{p^{g_{1}g_{2}}}^2$ via right multiplication by $\twomat{0}{-1}{1}{0}$.
\end{itemize}

Suppose now that $p, q \in \Gamma$ both have at least three non-zero coordinates and their associated fat point modules
are isomorphic, thus $M_p^2 \cong M_{q}^2$. By Proposition \ref{prop: fatpointsotherwaygen}, we can take direct sums of
these fat point modules and consider them as right $M_2(A^{G,\mu})$-modules. On restricting them down to the subalgebra
$A$ we obtain via \eqref{eq: orbitdecomp} the following isomorphism:
\begin{equation}\label{eq: orbitdecomp2}
\bigoplus_{g \in G} M_{p^{g}} \cong \bigoplus_{g \in G} M_{q^{g}}.
\end{equation}

Both modules in \eqref{eq: orbitdecomp2} are f.g.\ $\N$-graded modules of GK dimension 1. By \cite[Proposition
1.5]{smith1992the} the factors in a critical composition series of such a module are, when considered in high degree,
unique up to permutation and isomorphism. In our case this implies that we must have $\pi(M_q) \cong \pi(M_{p^g})$ for
some $g \in G$. Thus there exists some $n \in \N$ such that $M_q[n]_{\geq 0}\cong M_{p^g}[n]_{\geq 0}$ in
$\text{grmod}(A)$. But $M_q[n]_{\geq 0} \cong M_{q^{\sigma^{n}}}$ for any point in $\Gamma$, in which case one has an
isomorphism in $\text{grmod}(A)$ of the form $M_{q^{\sigma^{n}}} \cong M_{(p^g)^{\sigma^{n}}}$. Since point modules over
$A$ are parameterised up to isomorphism by the closed points of $\Gamma$ and $\sigma$ is an automorphism, it follows that $q=p^g$.
\end{proof}

\section{Twists of 4-dimensional Sklyanin algebra}\label{sec: twist4dimsklyanin}
In this section we focus our attention on twists of 4-dimensional Sklyanin algebras. We begin in \S\ref{subsec: applyprops} by applying some machinery from previous work \cite{davies2014cocycle1} to show that such twists have many good properties. Furthermore, we show that our later study of one particular twist encapsulates many other twists to which it is isomorphic up to a change of parameters. In \S\ref{subsec: modulesoversklytwist} we apply the results of \S\ref{sec: modules} to show that the point scheme of one such twist consists of precisely 20 points, as well as describing its fat point modules of multiplicity 2. Since the point scheme parameterises point modules, we show that there are 20 isomorphism classes of them. 

Let us fix the notation $A:= A(\alpha, \beta, \gamma)$, thus $A$ denotes a 4-dimensional Sklyanin algebra as defined in \S\ref{subsec: 4dimsklyanin}. We will omit the parameters if no ambiguity will arise. We also fix the group $G := (C_2)^2 = \langle g_1, g_2 \rangle$, the 2-cocycle $\mu$ defined in Hypotheses \ref{hyp: genhypforfatpts}, and the isomorphism $G \cong G^{\vee}$ given by $g \mapsto \chi_g$, where 
\begin{equation*}
\chi_g(h) = \left\{ \begin{array}{cl} 1 & \text{if }g=e\text{ or }h \in \{e, g\}, \\ -1 & \text{otherwise}, \end{array}\right. 
\end{equation*}
for all $g, h \in G$. 

Recall from \S\ref{subsec: 4dimsklyanin} that associated to $A$ there is an elliptic curve $E$, along with automorphism of the curve $\sigma$. We will assume that the following hypothesis holds throughout the remainder of the paper.
\begin{hypsing}\label{hypsing: siginforder}
The automorphism $\sigma$ has infinite order.
\end{hypsing}

\subsection{The twist and its properties}\label{subsec: applyprops}
In this section we apply the results of \cite{davies2014cocycle1} to 4-dimensional Sklyanin algebras. Let us begin by defining a $G$-grading on $A(\alpha,\beta,\gamma)$ under which the standard generators are homogeneous and lie in the following components:
\begin{equation}\label{eq: 4sklyaningrading}
x_0 \in A(\alpha,\beta,\gamma)_{e},\; x_1 \in A(\alpha,\beta,\gamma)_{g_{1}},\; x_2 \in A(\alpha,\beta,\gamma)_{g_{2}},\; x_3 \in A(\alpha,\beta,\gamma)_{g_{1}g_{2}}.
\end{equation}

The algebra $A(\alpha,\beta,\gamma)^{G,\mu}$ has generators $v_0,v_1,v_2,v_3$ and the following six quadratic relations:
\begin{equation}\label{eq: twistrelns}
\begin{array}{ll}
f_1^{\mu}:=[v_0,v_1]-\alpha[v_2,v_3], &  f_2^{\mu}:=[v_0,v_1]_{+} -[v_2,v_3]_{+}, \\ \relax
f_3^{\mu}:=[v_0,v_2]-\beta[v_3,v_1], &  f_4^{\mu}:=[v_0,v_2]_{+} -[v_3,v_1]_{+}, \\ \relax
f_5^{\mu}:=[v_0,v_3]+\gamma[v_1,v_2], &  f_6^{\mu}:=[v_0,v_3]_{+} +[v_1,v_2]_{+}.\end{array}
\end{equation}
We will keep this notational convention going forward; generators of the Sklyanin algebra $A$ will be denoted by $x_i$, whereas those of a cocycle twist of it, $A^{G,\mu}$, will be denoted by $v_i$.

Results from our previous work lead to the following theoerem.
\begin{theorem}\label{theorem: 4sklytwistprops}
Assume that the parameter triple $(\alpha, \beta, \gamma)$ satisfies \eqref{eq: 4sklyanincoeffcond} and is not equal to either $(-1,1,\gamma)$, $(\alpha,-1,1)$ or  $(1,\beta,-1)$. Then the algebra $A(\alpha,\beta,\gamma)^{G,\mu}$ is finitely generated in degree 1 and has the following properties:
\begin{itemize}
 \item[(i)] it is a strongly noetherian;
 \item[(ii)] it has Hilbert series $(1-t)^{-4}$;
 \item[(iii)] it is AS-regular domain of global dimension 4;
 \item[(iv)] it is Auslander regular;
 \item[(v)] it satisfies the Cohen-Macaulay property;
 \item[(vi)] it is Koszul.
\end{itemize}
\end{theorem}
\begin{proof}
To see that $A(\alpha,\beta,\gamma)^{G,\mu}$ is finitely generated in degree 1, we use the fact that $A(\alpha,\beta,\gamma)$ has this property in conjunction with \cite[Lemma 4.3]{davies2014cocycle1}. By \cite[Corollary 4.12]{artin1999generic} $A(\alpha,\beta,\gamma)$ is strongly noetherian, whence by \cite[Corollary 4.11]{davies2014cocycle1} so is $A(\alpha,\beta,\gamma)^{G,\mu}$.

By \cite[Thm 5.5]{smith1992regularity} and \cite[Corollary 1.9]{levasseur1993modules} $A(\alpha,\beta,\gamma)$ has all of the remaining properties except being a domain. We can therefore use the following results from \cite{davies2014cocycle1}: Lemma 4.9 for (ii), Corollary 4.23 and Proposition 4.14 for (iii), Proposition 4.32 for (iv) and (v) and Proposition 4.25 for (vi) to show that the corresponding properties are preserved under twisting.  

To complete the proof, observe that since $A(\alpha,\beta,\gamma)^{G,\mu}$ is AS-regular of global dimension 4 it must be a domain by \cite[Theorem 3.9]{artin1991modules}.
\end{proof}

\label{subsec: permuteaction}
We next consider the twist with relations \eqref{eq: twistrelns} in respect of \cite[\S 3.2]{davies2014cocycle1}, where we considered group actions on $G$-gradings/cocycles.
\begin{proposition}\label{lem: 24to1}
Let $G$ be the Klein-four group and $(\alpha,\beta,\gamma)$ a parameter triple satisfying \eqref{eq: 4sklyanincoeffcond}. There are 24 actions of $G$ on $A(\alpha,\beta,\gamma)$ by $\N$-graded algebra automorphisms for which the following hold:
\begin{itemize}
 \item[(i)] $G$ acts diagonally on the generators $x_0,x_1,x_2$ and $x_3$;
 \item[(ii)] the action of $G$ on $A(\alpha,\beta,\gamma)_1$ affords the regular representation.
\end{itemize}
For any such action of $G$, consider the cocycle twist $A(\alpha,\beta,\gamma)^{G,\mu}$. All such algebras are isomorphic up to a change of parameters which still satisfy \eqref{eq: 4sklyanincoeffcond}.
\end{proposition}
\begin{proof}
We identify each $G$-graded component of $A$ with the corresponding index of the generator it contains in \eqref{eq: 4sklyaningrading}. In this manner, any other grading corresponding to an action of $G$ affording the regular
representation gives rise to a permutation in the symmetric group $S_4$. It is trivial to check using the relations in \eqref{eq: 4sklyaninrelnsintro} that
each permutation in $S_4$ corresponds to a genuine $G$-grading.

The action which induces the grading in \eqref{eq: 4sklyaningrading}
corresponds to the identity permutation, hence \emph{for this proof only the associated twist will be denoted
$A(\alpha,\beta,\gamma)^{G,\mu;(id)}$}. Our aim is therefore to understand
$A(\alpha,\beta,\gamma)^{G,\mu;\sigma}$ for other permutations $\sigma \in S_4$. 

Let us now assume that any $G$-grading is one of the 24 arising from an action of $G$ on the generators by the regular
representation. Note that $\text{Aut}(G) \cong S_3$, where automorphisms permute the order 2 elements. 

Two $G$-gradings are twists of each other by an automorphism of $G$ if and only if the identity components of their $G$-gradings are equal. Recast in terms of permutations, we obtain a partition of $S_4$ by the subsets 
\begin{equation*}
H_j=\{\sigma \in S_4 : \; \sigma^{-1}(0)=j\} \;\text{ for }\;j=0,1,2,3. 
\end{equation*}
Let us choose the identity map and the transpositions $(0j)$ for $j = 1,2,3$ as representatives of these subsets. 

We now show that the action of $\text{Aut}(G)$ on 2-cocycles of $G$ is trivial. To do this we first identify the subgroup of $S_4$ given by
\begin{equation}\label{eq: 0stabperms}
H_0=\{(id),(12),(23),(13),(123),(132)\},
\end{equation}
with $\text{Aut}(G)$. The identification we use arises naturally from our prior identification of elements of $G$ with the set $\{0,1,2,3\}$ using \eqref{eq: 4sklyaningrading}.

Observe that the following table describes for each permutation $\sigma \in H_0$ how the 2-cocycles $\mu$ and $\mu^{(\sigma^{-1})}$ are cohomologous via the function $\rho: G \rightarrow k^{\times}$, where $i\in k$ satisfies $i^2=-1$:
\begin{equation}\label{eq: sigmaactcoboundary}
\begin{array}{c|cccc}
\sigma & \rho(e) & \rho(g_1) & \rho(g_2) & \rho(g_1g_2) \\
\hline
(12) & 1 & -1 & 1 & 1 \\ 
(13) & 1 & i & 1 & i \\ 
(23) & 1 & 1 & i & i \\ 
(123) & 1 & i & -1 & i \\ 
(132) &  1 & 1 & i & -i 
\end{array}
\end{equation}

Now consider $\sigma \in H_j$. We already know that the $G$-grading associated to $\sigma$ is obtained from that associated to $(0j)$ by twisting the grading by an automorphism of $G$. But \eqref{eq: sigmaactcoboundary} shows that
the action of $\text{Aut}(G)$ on 2-cocycles is trivial. Thus by \cite[Proposition 3.5]{davies2014cocycle1} the twists $A(\alpha,\beta,\gamma)^{G,\mu;\sigma}$ and $A(\alpha,\beta,\gamma)^{G,\mu;(0j)}$ must be isomorphic.

The rows of Table \ref{table: twistisos} each give a grading on $A(\alpha, \beta, \gamma)$ for generators $x=(x_0,x_1,x_2,x_3)$ corresponding to some transposition $(0j)$. By twisting this grading using $\mu$ one obtains an algebra isomorphic to $A(\alpha', \beta', \gamma')^{G,\mu}$ as given in the last column, where the middle column gives the scaling of $v=(v_0,v_1,v_2,v_3)$ necessary to realise the isomorphism. 
\begin{equation}\label{table: twistisos}
\begin{array}{ c | c | c }
 \text{Grading on }x  & \text{Isomorphism, }v \cdot -^{T} & \text{Result} \\ \hline\noalign{\smallskip}
(g_1,e,g_2,g_1g_2) & \left(1,\frac{i}{\sqrt{\beta \gamma}},-\frac{1}{\sqrt{\gamma}},-\frac{i}{\sqrt{\beta}}\right) & A\left(\alpha, \frac{1}{\beta}, \frac{1}{\gamma}\right)^{G,\mu} \\

(g_2,g_1,e,g_1g_2) & \left(1,\frac{i}{\sqrt{\gamma}},\frac{i}{\sqrt{\alpha\gamma}},\frac{1}{\sqrt{\alpha}}\right) & 
A\left(\frac{1}{\alpha}, \beta, \frac{1}{\gamma}\right)^{G,\mu} \\
  
(g_1g_2,g_1,g_2,e) & \left(1,\frac{i}{\sqrt{\beta}},\frac{1}{\sqrt{\alpha}},\frac{i}{\sqrt{\alpha\beta}}\right) & 
A\left(\frac{1}{\alpha}, \frac{1}{\beta}, \gamma\right)^{G,\mu}
\end{array}
\end{equation}

If the parameters $(\alpha, \beta, \gamma)$ satisfy \eqref{eq: 4sklyanincoeffcond} then so does the triple obtained by taking the reciprocal of two of the three parameters or permuting them.
\end{proof}
Thus our study of $A^{G,\mu}$, whose relations are those in \eqref{eq: twistrelns}, will encompass many other twists up to a change of parameters.

\subsection{Modules over $A^{G,\mu}$}\label{subsec: modulesoversklytwist}
In this section we apply the results from \S\ref{sec: modules} to $A$ and $A^{G,\mu}$. The main result of this section is the following theorem.
\begin{theorem*}[{Theorem \ref{thm: finitepointscheme}}]
Suppose that $A$ is associated to the elliptic curve $E$ and automorphism $\sigma$, which has infinite order. Then the point scheme $\Gamma$ of the cocycle twist $A^{G,\mu}$ consists of 20 points.
\end{theorem*}
It was shown in \cite[Proposition 2.4 and 2.5]{smith1992regularity} that point modules over $A$ are parameterised by the elliptic curve $E \subset \proj{k}{3}$ and four additional points $e_j$ which are defined in \eqref{eq: 4sklyanineipts}. Furthermore, in \cite[Corollary 2.8]{smith1992regularity} it is shown that $\sigma$ fixes the four exceptional points and is given by a translation on $E$. It is therefore clear that Theorem \ref{thm: finitepointscheme} shows that the geometry associated to $A^{G,\mu}$ is markedly different. 

Our proof of the main theoerem requires a series of results that in particular give a detailed description of the (fat) point modules over $A^{G,\mu}$. We begin with the following lemma.
\begin{lemma}\label{lem: multilins}
Consider the multilinearisations of the relations of $A^{G,\mu}$
given in \eqref{eq: twistrelns}:
\begin{gather}
\begin{aligned}\label{eqn: multilins}
&m_1:=v_{01}v_{12}-v_{11}v_{02}-\alpha v_{21}v_{32}+\alpha v_{31}v_{22},\\
&m_2:=v_{01}v_{12}+v_{11}v_{02}-v_{21}v_{32}-v_{31}v_{22}, \\
&m_3:=v_{01}v_{22}-v_{21}v_{02}+\beta v_{11}v_{32}-\beta v_{31}v_{12},\\
&m_4:=v_{01}v_{22}+v_{21}v_{02}-v_{31}v_{12}-v_{11}v_{32}, \\
&m_5:=v_{01}v_{32}-v_{31}v_{02}+\gamma v_{11}v_{22}-\gamma v_{21}v_{12},\\
&m_6:=v_{01}v_{32}+v_{31}v_{02}+v_{11}v_{22}+v_{21}v_{12}. 
\end{aligned}
\end{gather}
The closed subscheme $\Gamma_2 \subset \proj{k}{3} \times \proj{k}{3}$ determined by the equations in \eqref{eqn:
multilins} is isomorphic to the graph of the point scheme $\Gamma$ under a scheme automorphism $\phi$. Thus
\begin{equation}\label{eq: 4sklyaninpointlocus}
\Gamma_2=\{q = (p,p^{\phi}) \in \proj{k}{3} \times \proj{k}{3}: m_i(q)=0\;\text{ for $i=1,\ldots,6$} \}.
\end{equation}
Furthermore, the closed points of $\Gamma$ parameterise point modules over $A^{G,\mu}$.
\end{lemma}
\begin{proof}
One can use Theorem \ref{theorem: 4sklytwistprops} to see that $A^{G,\mu}$ satisfies the hypotheses of \cite[Theorem 1.4]{shelton1999embedding}, from which the result follows.
\end{proof}
 
Before the next lemma, let us introduce some notation for certain points in $\proj{k}{3}$:
\begin{equation}\label{eq: 4sklyanineipts}
e_0:=(1,0,0,0),\;\; e_1:=(0,1,0,0),\;\; e_2:=(0,0,1,0),\;\; e_3:=(0,0,0,1).
\end{equation}
Let us also define $\pi_1: \proj{k}{3} \times \proj{k}{3} \rightarrow \proj{k}{3}$ to be projection to the first coordinate. 
\begin{lemma}\label{lem: ptschemecontains}
The scheme $\Gamma_2$ contains the closed points $(e_j,e_j)$ for $j = 0,1,2,3$, as well as the following 16 points, where $i\in k$ satisfies $i^2= -1$:
\begin{gather}
\begin{aligned}\label{eq: 4sklyanintwistpts}
&((1,\pm i, \pm i ,1),(1, \pm i, \pm i , 1)),\;\;\; ((1,\pm i, \mp i ,-1),(1, \pm i, \mp i , -1)),\\
&\left(\left(1,-(\beta \gamma)^{-\frac{1}{2}},\mp \gamma^{-\frac{1}{2}},\mp\beta^{-\frac{1}{2}} \right),\left(1,-(\beta
\gamma)^{-\frac{1}{2}},\pm \gamma^{-\frac{1}{2}},\pm\beta^{-\frac{1}{2}} \right)\right), \\
&\left(\left(1,(\beta \gamma)^{-\frac{1}{2}},\mp \gamma^{-\frac{1}{2}},\pm\beta^{-\frac{1}{2}} \right),\left(1,(\beta
\gamma)^{-\frac{1}{2}},\pm \gamma^{-\frac{1}{2}},\mp\beta^{-\frac{1}{2}} \right)\right), \\
&\left(\left(1,\pm i\gamma^{-\frac{1}{2}}, (\alpha \gamma)^{-\frac{1}{2}},\pm i \alpha^{-\frac{1}{2}} \right),\left(1,
\mp i\gamma^{-\frac{1}{2}},(\alpha \gamma)^{-\frac{1}{2}},\mp i \alpha^{-\frac{1}{2}} \right)\right), \\
&\left(\left(1, \mp i\gamma^{-\frac{1}{2}},- (\alpha \gamma)^{-\frac{1}{2}},\pm i \alpha^{-\frac{1}{2}} \right),\left(1,
\pm i\gamma^{-\frac{1}{2}},- (\alpha \gamma)^{-\frac{1}{2}},\mp i \alpha^{-\frac{1}{2}} \right)\right), \\
&\left(\left(1,\pm \beta^{-\frac{1}{2}},\pm i \alpha^{-\frac{1}{2}},i (\alpha \beta)^{-\frac{1}{2}} \right),\left(1,\mp
\beta^{-\frac{1}{2}},\mp i \alpha^{-\frac{1}{2}}, i (\alpha \beta)^{-\frac{1}{2}} \right)\right), \\
&\left(\left(1,\pm \beta^{-\frac{1}{2}},\mp i \alpha^{-\frac{1}{2}},- i (\alpha \beta)^{-\frac{1}{2}} \right),\left(1,
\mp \beta^{-\frac{1}{2}},\pm i \alpha^{-\frac{1}{2}},- i (\alpha \beta)^{-\frac{1}{2}} \right)\right).
\end{aligned}
\end{gather}
In particular, $\Gamma$ contains at least 20 points.
\end{lemma}
\begin{proof}
It is easy to see that $(e_j,e_j) \in \Gamma_2$ for $j = 0,1,2,3$, since these points satisfy the multilinearisations in \eqref{eqn: multilins}. A routine verification---using \eqref{eq: 4sklyanincoeffcond} where necessary---confirms that the points in \eqref{eq: 4sklyanintwistpts} also satisfy the
equations in \eqref{eqn: multilins}.

To see that the 20 points are distinct, note that by \eqref{eq: 4sklyaninpointlocus} it suffices to show that the points in $\pi_1(\Gamma)$ are distinct, which follows from \eqref{eq: 4sklyanincoeffcond}.
\end{proof}

We define $\Gamma'$ to be the set of points of the form $\pi_1(q)$ for $q$ in \eqref{eq: 4sklyanintwistpts}. Thus $|\Gamma'|=16$, with Lemma \ref{lem: ptschemecontains} implying that $\Gamma' \cup \{e_0,e_1,e_2,e_3\} \subseteq \Gamma$.

\begin{lemma}\label{lem: phiaut}
The scheme automorphism $\phi$ defined by \ref{lem: multilins} has order 2. In particular, it fixes 8 of the points given in Lemma \ref{lem: ptschemecontains}.
\end{lemma}
\begin{proof}
Notice that the multilinearisations in \eqref{eqn: multilins}
are invariant under the map $v_{i1} \leftrightarrow v_{i2}$. Thus $\Gamma_2$ is invariant under the automorphism which
switches components of the ambient space $\proj{k}{3} \times \proj{k}{3}$. 

We can conclude using \eqref{eq: 4sklyaninpointlocus} that $(p,p^{\phi}) \in \Gamma_2$ if and only if $(p^{\phi},p) \in \Gamma_2$. But since $\Gamma_2$ is the graph of $\Gamma$ under $\phi$, we must have $p=p^{\phi^{2}}$. Consequently, $\phi$ has order at most 2. To see that its order is precisely 2, observe that there are points listed in Lemma \ref{lem: ptschemecontains} which are not fixed.
\end{proof}

As we will now begin to work with modules over $A$ and $A^{G,\mu}$ simultaneously, it will be useful to introduce some notation for clarity. Let point modules over $A$ be denoted by $M_p$, while those over $A^{G,\mu}$ be denoted by $\widetilde{M}_p$.

In order to be able to apply the results from \S\ref{sec: modules} we need to check that both $A$ and $A^{G,\mu}$ satisfy Hypotheses \ref{hyp: genhypforfatpts}. To verify this one can use Theorem \ref{theorem: 4sklytwistprops}, as well as noting that the action of $G$ on $A^{G,\mu}$ is chosen to be that which induces the $G$-grading inherited from $A$.

The next lemma is straightforward, but very useful.
\begin{lemma}\label{lem: threegensannihilate}
For all $p \in E$ at least three coordinates of $p$ are non-zero. Furthermore, the action of $G$ on the point scheme of
$A$ restricts to $E$.
\end{lemma}
\begin{proof}
Assume that $p= (p_0,p_1,p_2,p_3) \in E$. We will use \cite[Proposition 2.5]{smith1992regularity}, which shows that $E$ is defined by the following two equations in $\proj{k}{3}$:
\begin{equation}\label{eq: thcroreextn}
\begin{gathered}
y_0^2+y_1^2+y_2^2+y_3^2 = 0,\\
y_3^2+\left(\frac{1-\gamma}{1+\alpha}
\right)y_1^2+\left(\frac{1+\gamma}{1-\beta}\right) y_2^2 = 0.
\end{gathered}
\end{equation}

The action of $G$ on the point scheme of $A$ is described in \eqref{eq: Gactonpoints}. Since the action only ever changes the sign of a component, it is clear from \eqref{eq: thcroreextn} that if $p \in E$ then $p^g \in E$ for all $g \in G$.

Since $e_j \notin E$ for $j=0,1,2,3$ we can assume that at least two
coordinates of $p$ are non-zero. If there were two non-zero
entries, $p_l$ and $p_m$ say, then the equations defining $E$ would reduce to the form 
\begin{equation*}
p_l^2+p_m^2=p_l^2+ \lambda p_m^2=0, 
\end{equation*}
for some $\lambda \in k$. The only solution when $\lambda \neq 1$ is $p_l=p_m=0$, which results in a contradiction. If
$\lambda=1$ then either $\frac{1-\gamma}{1+\alpha}=1$ or $\frac{1+\gamma}{1-\beta}=1$ or
$\frac{1-\gamma}{1+\alpha}=\frac{1+\gamma}{1-\beta}$. 
In either case some simple manipulation shows that \eqref{eq: 4sklyanincoeffcond} is contradicted, which completes the proof.
\end{proof}

We will use $[p]$ to denote the $G$-orbit of a point $p \in \proj{k}{3}$; when necessary it will be made clear  whether this point lies on $E$ or in $\Gamma$. 
\begin{corollary}\label{cor: orbit4points}
For $p\in E$, $[p]$ contains precisely 4 points.
\end{corollary}
\begin{proof}
Suppose that $[p]$ contains fewer than 4 points. Then we must have $p^g=p^h$ for some distinct $g,h \in G$. Looking at any pair of the four points in \eqref{eq: Gactonpoints}, one can see that they have the same entry in precisely two coordinates, while the other two coordinates differ by a minus sign. If such a pair defined the same point then at least two of the entries would have to be zero. However, by Lemma \ref{lem: threegensannihilate} any point on $E$ must have at least 3 non-zero coordinates.
\end{proof}

The parameterisation described in the following result also has a geometric interpretation, as can be seen from
Proposition \ref{prop: fatpointsincohE}.
\begin{proposition}\label{prop: sklyaninfatpoints}
$A^{G,\mu}$ has a family of fat point modules of multiplicity 2 parameterised up to isomorphism by the $G$-orbits of $E$.
\end{proposition}
\begin{proof}
Let $p \in E$. By Lemma \ref{lem: threegensannihilate} at least three coordinates of $p$ are non-zero and the action of
$G$ on the point scheme of $A$ preserves $E$. Thus one may apply Proposition \ref{prop: fatpoints} to obtain a fat point module
$M_p^2$ over $A^{G,\mu}$. By Corollary \ref{cor: fatpointisoclasses} the only isomorphisms between such modules in
$\text{grmod}(A^{G,\mu})$ are of the form $M_p^2 \cong M_{p^g}^2$ for $g \in G$. 
\end{proof}
\begin{rem}\label{rem: qgrisoslater}
The right $M_2(A)$-module $M^2_{e_{j}}$ is isomorphic to $\widetilde{M}_{e_{j}}^2$ upon restriction to $A^{G,\mu}$. 
\end{rem}

For all $p \in E$ define the right $A^{G,\mu}$-module $\widetilde{F}_{[p]}:= M_{p}^2$, which is well-defined by Proposition \ref{prop: sklyaninfatpoints}.

Let us now describe the orbits in $\Gamma$.
\begin{lemma}\label{lem: gammaorbits}
Consider $\Gamma' \cup \{e_0,e_1,e_2,e_3\} \subseteq \Gamma$. Under the action of $G$ on $\Gamma$, $\Gamma' \cup \{e_0,e_1,e_2,e_3\}$ decomposes as the union
of 8 $G$-orbits; four singleton orbits  
\begin{equation*}
[e_0],\;\;[e_1],\;\;[e_2],\;\;[e_3],
\end{equation*}
and four orbits of order 4,
\begin{align*}
o_{e} &= \left[(1,i,i,1)\right], & o_{g_1} &= \left[\left(1,(\beta \gamma)^{-\frac{1}{2}}, -\gamma^{-\frac{1}{2}},\beta^{-\frac{1}{2}}
\right)\right],\\ 
o_{g_2} &= \left[\left(1, i\gamma^{-\frac{1}{2}}, (\alpha \gamma)^{-\frac{1}{2}}, i \alpha^{-\frac{1}{2}}
\right)\right], & o_{g_1g_2} &= \left[\left(1, \beta^{-\frac{1}{2}}, i \alpha^{-\frac{1}{2}},i (\alpha \beta)^{-\frac{1}{2}} \right)\right].
\end{align*}
Furthermore, if $[p]$ is an order 4 orbit then there exists $h \in G$ such that $(p^g)^{\phi}=(p^{g})^h$ for all $g \in G$. That is, the restriction of $\phi$ to each orbit of order 4 coincides with the action of a particular element of $G$. 
\end{lemma}
\begin{proof}
One can use \eqref{eq: Gactonpoints} to verify that the orbits under the action of $G$ are as stated above. The order 4 orbits are labelled such that the restriction of $\phi$ to the orbit $o_h$ coincides with the action of $h$. For example, $[(1,i,i,1)]$ is fixed pointwise by $\phi$ and therefore $o_e = [(1,i,i,1)]$. It is easy to check that the group elements corresponding to the other orbits are as stated. 
\end{proof}

In \cite{smith1993irreducible} Smith and Staniszkis classify fat point modules of all multiplicities over $A$ when $|\sigma| = \infty$. As the final remark in \S 4 op. cit. states, their work shows that there are four non-isomorphic fat
point modules of each multiplicity $e \geq 2$. These modules can be labelled
$F(\omega^{\sigma^{e-1}})$ for some 2-torsion point $\omega \in E$ in a natural way. 

\begin{proposition}\label{prop: sklyaninfatpts}
Consider $A$ as the invariant subring $M_2(A^{G,\mu})^G$. Four isomorphism classes of fat point modules of multiplicity
2 over $A$ arise as the restriction of modules of the form $\widetilde{M}_p^2$, where $\widetilde{M}_p$ is a point
module over $A^{G,\mu}$. When $|\sigma|=\infty$ one recovers in this manner all four fat point modules of multiplicity 2
over $A$.
\end{proposition}
\begin{proof}
First, note that $A^{G,\mu}$ satisfies Hypotheses \ref{hyp: genhypforfatpts}. Consider the 16 points in $\Gamma'$; by our assumption on parameters from \eqref{eq: 4sklyanincoeffcond}, each of these points has at least three non-zero coordinates. Thus by Proposition \ref{prop: fatpoints} one can
construct 16 fat point modules of multiplicity 2 over $A$, of the form $\widetilde{M}_p^2$ for $p \in \Gamma$. 

The 16 points we are considering are partitioned into the four $G$-orbits described in Lemma \ref{lem: gammaorbits}. Thus by
Corollary \ref{cor: fatpointisoclasses} there are precisely four isomorphism classes of such fat point modules. When
$|\sigma|=\infty$, the work in \cite{smith1993irreducible} shows that there are precisely four isomorphism classes of
fat point modules of multiplicity 2. Thus, under that hypothesis we recover each of these isomorphism classes.
\end{proof}
\begin{rem}\label{rem: degenfatptno}
The right $M_2(A^{G,\mu})$-module $\widetilde{M}^2_{e_{j}}$ becomes isomorphic to the direct sum $M_{e_{j}}^2$ upon restriction to $A$. In this respect $A^{G,\mu}$ behaves like $A$ (see Remark \ref{rem: qgrisoslater}).
\end{rem}

Proposition \ref{prop: sklyaninfatpts} allows us to associate a 2-torsion point
$\omega_{[p]} \in E$ to each $p \in \Gamma'$. Thus, let $F(\omega_{[p]}^{\sigma})$ denote the fat point module over $A$ which is isomorphic to the restriction of the $M_2(A^{G,\mu})$-module
$\widetilde{M}_{p^{g}}^2$ for all $g \in G$.

We have seen in Propositions \ref{prop: sklyaninfatpoints} and \ref{prop: sklyaninfatpts} that by taking direct sums of point modules over $A$ and $A^{G,\mu}$ we can obtain fat point modules of multiplicity 2 over the other algebra. The next result shows that one can repeat this process to recover point modules.
\begin{proposition}\label{prop: fatpointsotherway}
Assume that $|\sigma| = \infty$.
\begin{itemize}
\item[(i)] Let $p \in E$ and $\omega \in E$ be a 2-torsion point. On restriction to a module over $A=M_2(A^{G,\mu})^G$, one has
$(\widetilde{F}_{[p]})^2 \cong \bigoplus_{g \in G}M_{p^{g}}$. 

\item[(ii)] Now let $p \in \Gamma'$. The $M_2(A)$-module $F(\omega_{[p]}^{\sigma})^2$ is isomorphic upon restriction to
$A^{G,\mu}$ to the direct sum $\bigoplus_{g \in G}\widetilde{M}_{p^{g}}$. In this manner one recovers all 16 point
modules of $A^{G,\mu}$ associated to points in $\Gamma'$.
\end{itemize}
\end{proposition}
\begin{proof}
One can prove both (i) and (ii) by a direct application of Proposition \ref{prop: fatpointsotherwaygen}. We elaborate
the argument with respect to the final statement of (ii). This holds because each $G$-orbit $[p] \subset \Gamma'$
corresponds uniquely to a fat point module $F(\omega_{[p]}^{\sigma})^2$ over $A$ by Proposition \ref{prop:
sklyaninfatpts}. Proposition \ref{prop: fatpointsotherwaygen} then implies that taking a direct sum of this fat point
module and restricting returns the $A^{G,\mu}$-module $\bigoplus_{g \in G}\widetilde{M}_{p^{g}}$.
\end{proof}

We are now ready to prove the main result of this section, namely that the point scheme of $A^{G,\mu}$ consists of precisely 20 points.
\begin{theorem}\label{thm: finitepointscheme}
If $|\sigma|= \infty$ then $\Gamma = \Gamma' \cup \{e_0,e_1,e_2,e_3\}$. 
\end{theorem}
\begin{proof}
Suppose that $\widetilde{M}_p$ is a point module over $A^{G,\mu}$, where $p \in \proj{k}{3}$. We claim that if $p \neq
e_j$ then at least three coordinates of $p$ are non-zero. To see this, consider the following manner of expressing the
multilinearisations in \eqref{eqn: multilins}:
\begin{equation}\label{eq: 4sklyaninmatrixform}
\begin{pmatrix} 
-v_{11} & v_{01} & \alpha v_{31} & -\alpha v_{21} \\
v_{11} & v_{01} & -v_{31} & -v_{21} \\ 
-v_{21} & -\beta v_{31} & v_{01} & \beta v_{11} \\
v_{21} & -v_{31} & v_{01} & -v_{11} \\
-v_{31} & -\gamma v_{21} & \gamma v_{11} & v_{01} \\
v_{31} & v_{21} & v_{11} & v_{01}  \end{pmatrix} \cdot \begin{pmatrix} v_{02} \\ v_{12} \\ v_{22} \\ v_{32}
\end{pmatrix} =0.
\end{equation}
By assumption $(p,p^{\phi}) \in \Gamma_2$. Since $\phi$
is an automorphism this is the unique point $q \in \Gamma_2$ for which $p=\pi_1(q)$. In particular, this implies that
when the coordinates of $p$ are substituted into the left-hand matrix in \eqref{eq: 4sklyaninmatrixform}, that matrix
must have rank less than or equal to 3. 

Suppose, therefore, that two coordinates of $p$ are zero. We may assume that the remaining two coordinates are non-zero, otherwise $p$ would be one of the four points of the form $e_j$. As $\{\alpha,\beta, \gamma\} \cap \{0,\pm 1\}=\emptyset$ the matrix in \eqref{eq: 4sklyaninmatrixform} has rank at least 4, regardless of which two coordinates of $p$ are zero. Thus we may assume that at least three coordinates of $p$ are non-zero.

By Proposition \ref{prop: fatpoints} $\widetilde{M}_p^2$ is a fat point
module of multiplicity 2 over $A$. But when $|\sigma|=\infty$ there are only four such modules up to isomorphism, hence we know that $\widetilde{M}_p^2$ must be isomorphic to one of them. Proposition \ref{prop: fatpointsotherway} implies
that there
exists some point $z \in \Gamma'$ for which $\widetilde{M}_p^2 \cong F(\omega_{[z]}^{\sigma})$ as right $A$-modules.

Applying Proposition \ref{prop: fatpointsotherwaygen} to $\widetilde{M}_p^2$ produces the right $A^{G,\mu}$-module
$\bigoplus_{g \in G}\widetilde{M}_{p^{g}}$. Using the decomposition of the right $A^{G,\mu}$-module
$(F(\omega_{[z]}^{\sigma}))^2$ from Proposition \ref{prop: fatpointsotherway}, we must have an isomorphism
\begin{equation}\label{eq: dualdecomp}
\bigoplus_{g \in G}\widetilde{M}_{p^{g}} \cong \bigoplus_{g \in G}\widetilde{M}_{z^{g}},
\end{equation}
of right $A^{G,\mu}$-modules. But by \cite[Proposition 1.5]{smith1992the} the factors in a critical composition series
of a f.g.\ $\N$-graded module of GK dimension 1 are unique up to permutation and isomorphism in high degree. That result applies to the modules in \eqref{eq: dualdecomp}, hence we may conclude that $\pi(\widetilde{M}_{p}) \cong
\pi(\widetilde{M}_{z^{g}})$ in $\text{qgr}(A^{G,\mu})$ for some $g \in G$. This implies the existence of $n \in \N$ such
that $\widetilde{M}_{p^{\phi^{n}}} \cong \widetilde{M}_{\left(z^{g}\right)^{\phi^{n}}}$ in $\text{grmod}(A^{G,\mu})$. Since
$\Gamma$ parameterises isomorphism classes of point modules over $A^{G,\mu}$ and $\phi$ is an automorphism, we may
conclude that $p= z^g$.
\end{proof}

\section{Twists of a factor ring}\label{sec: thcr}
In this section we will study a cocycle twist of an important factor ring of the 4-dimensional Sklyanin algebra $A=A(\alpha,\beta,\gamma)$. Let us begin by describing the main object of study. By \cite[Corollary 3.9]{smith1992regularity} $A$ contains two central elements of degree 2, with such elements having the form
\begin{equation}\label{eq: 4sklyanincentre}
 \Omega_1:=-x_0^2+x_1^2+x_2^2+x_3^2,\;\; \Omega_2:=x_1^2+\left(\frac{1+\alpha}{1-\beta}
\right)x_2^2+\left(\frac{1-\alpha}{1+\gamma} \right)x_3^2.
\end{equation}

Recall from \S\ref{subsec: 4dimsklyanin} that associated to $A$ there is an elliptic curve $E$ and automorphism $\sigma$ of $E$ given by translation by a point. Corollary 3.9 from \cite{smith1992regularity} shows that there is an isomorphism 
\begin{equation*}
\frac{A}{(\Omega_1,\Omega_2)} \cong B(E,\calL,\sigma),
\end{equation*}
where the right-hand side is a twisted homogeneous coordinate ring and $\calL:=\calO_E(1)$ a very ample invertible sheaf with 4 global sections that gives an embedding $E \hookrightarrow \proj{k}{3}$. The ring $B(E,\calL,\sigma)$ will be denoted by $B$ whenever there is no ambiguity in doing so.

The twisting operation that we applied to $A$ in \S\ref{sec: twist4dimsklyanin} restricts to the factor ring $B$, whence we obtain a cocyle twist $B^{G, \mu}$. We will show that this is a prime ring but not a domain (see Corollary \ref{cor: actuallyprime} and Remark \ref{rem: bgmunilp} respectively). However, $B^{G, \mu}$ can still be described as a suitable generalisation of a twisted homogeneous coordinate ring (see \S\ref{subsubsec: geomdescrthcrtwist} for definitions and the formal result). This means that, like $B$, its structure is strongly controlled by geometric data, although for $B^{G, \mu}$ this involves another elliptic curve $E^G$. This is the curve determined by the orbit space of $E$ under a natural action of $G$. The geometric description allows us to ascertain many of the properties of $B^{G,\mu}$.

Throughout this section we will continue assuming that Hypothesis \ref{hypsing: siginforder} holds, thus the automorphism $\sigma$ has infinite order. Since $\sigma$ is given by a translation by a point, under this assumption there are no points on $E$ fixed by $\sigma$. 

\subsection{Properties of the twist}\label{subsec: propthcrtwist}
We begin this section by showing that the construction of $B^{G,\mu}$ described in the introduction to \S\ref{sec: thcr} is valid, followed by giving some of its elementary properties.
\begin{lemma}\label{lem: twistthcr}
The degree 2 elements
\begin{equation}\label{eq: 4sklyanintwistcentre}
\Theta_1:=-v_0^2+v_1^2+v_2^2-v_3^2,\;\; \Theta_2:=v_1^2+\left(\frac{1+\alpha}{1-\beta}
\right)v_2^2-\left(\frac{1-\alpha}{1+\gamma} \right)v_3^2,
\end{equation}
in $A^{G,\mu}$ are central and form a regular sequence. Moreover, there is an isomorphism of $k$-algebras
\begin{equation}\label{eq: isofactors}
B^{G,\mu}:=B(E,\calL,\sigma)^{G,\mu}=\left(\frac{A(\alpha,
\beta,\gamma)}{(\Omega_1,\Omega_2)} \right)^{G,\mu} \cong \frac{A(\alpha, \beta,\gamma)^{G,\mu}}{(\Theta_1,\Theta_2)}.
\end{equation}
In particular, $B^{G,\mu}$ is generated in degree 1.
\end{lemma}
\begin{proof}
Consider the grading defined on $A$ by \eqref{eq: 4sklyaningrading}. Since $\Omega_1, \Omega_2 \in A_2 \cap A_e$ the cocycle twist performed in the previous section can also be applied to $B$. By \cite[Theorem 5.4]{smith1992regularity}, $\Omega_1$ and $\Omega_2$ form a regular sequence of central elements in $A$. Applying \cite[Lemma 4.1]{davies2014cocycle1} to successive factor rings shows that this remains true for $\Theta_1$ and $\Theta_2$ in $A^{G,\mu}$. Now apply \cite[Lemma 4.2]{davies2014cocycle1} to $A(\alpha, \beta,\gamma)$ and the ideal $(\Omega_1,\Omega_2)$, which proves that the isomorphism in \eqref{eq: isofactors} holds. 
\end{proof}

We will now prove that the twist has several further properties. Note that the result holds even when Hypothesis \ref{hypsing: siginforder} is not assumed.
\begin{theorem}\label{thm: thcrtwistprops}
$B^{G,\mu}$ has the following properties:
\begin{itemize}
 \item[(i)] it is strongly noetherian;
 \item[(ii)] it has GK dimension 2;
 \item[(iii)] it is Auslander-Gorenstein of dimension 2;
 \item[(iv)] it satisfies the Cohen-Macaulay property;
 \item[(v)] it is Koszul.
\end{itemize}
\end{theorem}
\begin{proof}
For (ii), observe that \cite[Lemma 4.9]{davies2014cocycle1} implies that GK dimension is preserved under twisting, while $\text{GKdim}(B)=2$ by \cite[Proposition 1.5]{artin1990twisted}. To prove (iv) we note that $A$ is Cohen-Macaulay
and therefore so is $B$ by \cite[Corollary 6.7]{levasseur1992some}. Applying \cite[Proposition 4.32(i)]{davies2014cocycle1} shows
that $B^{G,\mu}$ is also Cohen-Macaulay. 

By Theorem \ref{theorem: 4sklytwistprops}, $A^{G,\mu}$ has properties (i) and (iii), therefore it suffices to show that
these two properties are preserved by going to the factor ring. Lemma \ref{lem: twistthcr} shows that $\Theta_1$ and
$\Theta_2$ form a regular sequence of central elements in $B^{G,\mu}$, in which case using \cite[Proposition
4.9(1)]{artin1999generic} shows that (i) is true. Applying \cite[Theorem 3.6(2)]{levasseur1992some} twice proves (iii).
Finally, $B$ is Koszul by \cite[Theorem 3.9]{stafford1994regularity}. Using \cite[Proposition 4.25]{davies2014cocycle1} shows that
this is also true for $B^{G,\mu}$.
\end{proof}
\begin{rem}\label{rem: bgmunilp}
Whereas $B$ is a domain---which follows from the irreducibility of $E$---$B^{G,\mu}$ contains nilpotent elements in
degree 1. Let $v=v_0-i v_1 - i v_2 -v_3$ where $i^2=-1$. One can compute by hand or using the computer program Affine that
\begin{equation}\label{eq: nilpotentelement}
v^2 = -\Theta_1-i f_{2}^{\mu}-i f_4^{\mu} - f_6^{\mu} = 0.
\end{equation}

Since $v^2$ belongs to the relations of $B^{G,\mu}$ (which are $G$-invariant), so must $(v^g)^2=(v^2)^g$ for all $g \in
G$. Thus the other elements in the same $G$-orbit as $v$ are also nilpotent. Curiously, these elements are linearly
independent and therefore they generate $A^{G,\mu}$ as an algebra. In that ring the square of each generator is central,
since only terms of the form $f_i^{\mu}$ vanish in \eqref{eq: nilpotentelement}.
\end{rem}

While it may not be a domain, $B^{G,\mu}$ is prime as we prove in Corollary \ref{cor: actuallyprime}. We will prove the weaker statement that $B^{G,\mu}$ is semiprime in the next result. The following
definition is needed beforehand: given a ring $R$ on which a finite group $G$ acts by ring automorphisms, we say that
$R$ has no \emph{additive $|G|$-torsion} if $|G|r \neq 0$ for all $r \in R$.
\begin{proposition}\label{prop: semiprime}
$B^{G,\mu}$ is semiprime. 
\end{proposition}
\begin{proof}
As it is a matrix ring over a domain, $M_2(B)$ is certainly semiprime. Note that $\text{char}(k)\nmid |G|$, hence $|G|
\in k$ acts faithfully on $M_2(B)$ and thus there is no additive $|G|$-torsion. Since $B^{G,\mu}=M_2(B)^G$, one can
apply \cite[Corollary 1.5(1)]{montgomery1980fixd} to obtain the result.
\end{proof}

The next result of this section will show that the centre of $B^{G,\mu}$ is trivial. As a consequence we can determine
the centre of $A^{G,\mu}$ as well. To do so we will need to use the following property.
\begin{defn}[{\cite[Introduction]{rogalski2006proj}}]\label{defn: projsimple}
A $k$-algebra $A$ is \emph{projectively simple} if for any two-sided ideal of $A$, $I$
say, one has $\text{dim}_k(A/I) < \infty$.  
\end{defn}
\begin{proposition}\label{prop: centrethcrtwist}
The centre of $B^{G,\mu}$ is trivial, that is $Z(B^{G,\mu})=k$.
\end{proposition}
\begin{proof}
Suppose that $f \in Z(B^{G,\mu})$. We can assume without loss of generality that $f$ is homogeneous with respect to both
the $\N$-grading and the $G$-grading: the homogeneous components (with respect to either grading) of a central element
must also be central. Assume further that it has strictly positive \N-degree. We can untwist the factor ring $B^{G,\mu}/(f)$ to obtain $B/(f')$, where $f'$ is normal by \cite[Lemma 4.1]{davies2014cocycle1}. Moreover, $f'$ is regular since $B$ is a domain. One may therefore apply
\cite[Lemma 5.7]{levasseur1992some}, which implies that $\text{GKdim }B/(f')=\text{GKdim }B -1$. By \cite[Proposition
1.5]{artin1990twisted}, $B$ has GK dimension 2, hence $\text{GKdim }B/(f')=1$. However, when $|\sigma| = \infty$ holds $B$ is projectively simple by \cite[Proposition 0.1]{rogalski2006proj}, therefore by definition any factor ring must be finite-dimensional. This contradiction
implies that $Z(B^{G,\mu})=k$.
\end{proof}
\begin{corollary}\label{cor: centretwistskly}
The centre of $A^{G,\mu}$ is $Z(A^{G,\mu})=k[\Theta_1,\Theta_2]$.
\end{corollary}
\begin{proof}
By the previous proposition $Z(B^{G,\mu}) = k$. This allows us to repeat the method of proof used in \cite[Proposition 6.12]{levasseur1993modules}.
\end{proof}

Corollary \ref{cor: centretwistskly} allows us to show that $A^{G,\mu}$ is not isomorphic to one of the previously
discovered examples of AS-regular algebras with 20 point modules. Since the proof requires the computer program Affine, the result holds only for generic parameters values.
\begin{proposition}\label{prop: new}
Assume that a 4-dimensional Sklyanin algebra $A(\alpha,\beta,\gamma)$ is associated to the elliptic curve $E$ and
automorphism $\sigma$, with $|\sigma|=\infty$. Then for generic parameters $\alpha,\beta$ and $\gamma$,
$A^{G,\mu}$ is not isomorphic to any of the previously studied examples in the literature of
AS-regular algebras of dimension 4 with 20 point modules.
\end{proposition}
\begin{proof}
As remarked in \cite[\S 4]{lebruyn1995central}, graded Clifford algebras are finite over their centre. By Corollary
\ref{cor: centretwistskly} the centre of $A^{G,\mu}$ has GK dimension 2, and one can then see from \cite[Proposition
5.5]{krause2000growth} that $A^{G,\mu}$ cannot be finite over its centre and thus is not a graded Clifford algebra. 

It therefore suffices to consider the algebras studied in Examples 5.1 and 5.2 from \cite{cassidy2010generlizations}, since they encompass the remaining examples in the literature that are not finite over their centre (see \cite{shelton2001koszul} for some discussion of this). Computations using Affine show that such algebras do not have 3 central elements of degree 4 (as $A^{G,\mu}$ does), which is sufficient to prove the result.
\end{proof}

Although $B$ has point modules parameterised by the elliptic curve $E$, if $|\sigma|= \infty$ then $B^{G,\mu}$ can have at most 20 point modules by Theorem \ref{thm: finitepointscheme} and Lemma \ref{lem: twistthcr}. We now prove a much stronger statement.
\begin{theorem}
\label{thm: bgmunoptmodules}
$B^{G,\mu}$ has no point modules.
\end{theorem}
\begin{proof}
Let $\Gamma''$ denote the point scheme of $B^{G,\mu}$ and suppose that it is non-empty. By Theorem \ref{thm:
thcrtwistprops}(i), $B^{G,\mu}$ is strongly noetherian. We can therefore use \cite[Theorem 1.1]{rogalski2008canonical}
to conclude that there exists a graded homomorphism 
\begin{equation}\label{eq: canmaptothcr}
B^{G,\mu} \rightarrow B(\Gamma'',\mathcal{M},\phi),
\end{equation}
that is surjective in high degree, where $\phi$ is an automorphism of $\Gamma''$ and $\mathcal{M}$ is a $\phi$-ample
invertible sheaf. We know that the point scheme of $B^{G,\mu}$ is 0-dimensional, therefore by \cite[Proposition
1.5]{artin1990twisted} the GK dimension of $B(\Gamma'',\mathcal{M},\phi)$ must be 1.

Let $I$ denote the kernel of the map in \eqref{eq: canmaptothcr}. Since it is surjective in high degree there must exist
$n \in \N$ such that 
\begin{equation*}
\left(\frac{B^{G,\mu}}{I}\right)_{\geq n} \cong B(\Gamma'',\mathcal{M},\phi)_{\geq n},
\end{equation*}
as $k$-vector spaces. Thus $B^{G,\mu}/I_{\geq n}$ also has GK dimension 1. 

In addition to the existence of the homomorphism in \eqref{eq: canmaptothcr}, \cite[Theorem 1.1]{rogalski2008canonical}
states that $I_{\geq n}$ consists of the elements of $B^{G,\mu}$ that annihilate all of its point modules. By Lemma
\ref{lem: annGinvariant} this ideal is $G$-invariant, therefore one can twist the $G$-grading on $B^{G,\mu}/I_{\geq n}$
by the 2-cocycle $\mu$ to obtain a factor ring $B/I'$ for some ideal $I'$. The
factor ring has GK dimension 1 by Lemma \ref{theorem: geomdescthcrtwist}(i). 

By \cite[Theorem 4.4]{artin1995noncommutative}, ideals of $B$ correspond to $\sigma$-invariant
closed subschemes of $E$. Since $\sigma$ is given by translation by a point of infinite order on $E$, there are no
nontrivial such subschemes. Thus $I'$ must equal $0$ or have finite
codimension in $B$. In either case we get a contradiction: in the former because $B$ has GK dimension 2, and in the
latter because $B/I'$ would be finite-dimensional.
\end{proof}
\begin{rem}\label{rem: otherwayzerodiv}
If $B^{G,\mu}$ were a domain then \cite[Thm 0.2]{artin1995noncommutative} would imply that there was an equivalence of categories $\text{qgr}(B^{G,\mu}) \simeq
\text{coh}(Y)$ for some projective curve $Y$. But then $B^{G,\mu}$ would have a family of pairwise non-isomorphic point modules parameterised by $Y$, contradicting the conclusion of Proposition \ref{thm: bgmunoptmodules}. This is another way to see that $B^{G,\mu}$ has zero-divisors, examples of which were exhibited in Remark \ref{rem: bgmunilp}.
\end{rem}

\subsection{Structure in relation to Artin-Stafford theory}\label{subsubsec: geomdescrthcrtwist}
In Remark \ref{rem: bgmunilp} we exhibited zero divisors in $B^{G,\mu}$. One consequence of this is that it cannot be a twisted homogeneous coordinate ring over an irreducible curve. However, there is a more general construction in \cite{artin2000semiprime} of a twisted homogeneous coordinate ring of an order over such a curve. As the main result of this section shows (Theorem \ref{thm: artinstaffordmain}), $B^{G,\mu}$ can be described in such terms.

The work in \cite{artin2000semiprime} and its predecessor \cite{artin1995noncommutative} form Artin and Stafford's classification of noncommutative projective curves. In algebraic terms, they describe the geometry associated to c.g.\ semiprime noetherian algebra of GK dimension 2. By Theorem \ref{thm: thcrtwistprops} and Proposition \ref{prop: semiprime}, $B^{G,\mu}$ has such properties, and so we can study it from the viewpoint of the classification in \cite{artin2000semiprime}. 

Since the machinery used in classification is quite complex, we will begin this section with some definitions of the constructions and geometry involved. Although a priori we only know that $B^{G,\mu}$ is semiprime, it will be shown Corollary \ref{cor: actuallyprime} that it is in fact a prime ring. Thus we will only describe the geometry needed in the prime case, since that will be sufficient for our needs. 

Our first task is to define the geometric rings that feature in the classification. To this end, consider a projective curve $Y$ over $k$, whose function field $K$ has transcendence degree 1 over the base field. Now let $T$ be a central simple $K$-algebra that is finite-dimensional over $K$. The structure sheaf $\calO_Y$ is a
subsheaf of the constant sheaf of rational functions on $Y$, namely $K$, therefore the sections of $\calO_Y$ lie inside
$K$. Since $K \subseteq T$ such sections are also contained in $T$ .
\begin{defn}[{\cite[pg. 75]{artin2000semiprime}}]\label{defn: lattice}
We say that $\mathcal{\calL}$ is an \emph{$\calO_Y$-lattice} in $T$ if it is a sheaf of
finitely generated $\calO_Y$-submodules of $T$ for which $\calL K = T$. 
\end{defn}

When the term lattice is used in future it will be in the sense of this definition. Given such a lattice, one
can define the \emph{left} and \emph{right orders} of $\calL$
by
\begin{equation*}
E(\calL)=\{\alpha \in T:\alpha \calL \subseteq \calL\}\;\text{ and }\;E'(\calL)=\{\alpha \in T: \calL \alpha\subseteq
\calL\},
\end{equation*}
respectively. 

A lattice is said to be \emph{invertible} if it is a locally free left $E(\calL)$-module
of rank 1. Lemma 1.10 from \cite{artin2000semiprime} implies that an invertible lattice is also a locally free right
$E'(\calL)$-module of rank 1. 

A crucial point for us is that for invertible lattices $\calL$ and $\mathcal{M}$ the product lattice $\calL\mathcal{M}$---where the product takes places inside $T$---is isomorphic to the tensor product $\calL \otimes_{E(\mathcal{M})}
\mathcal{M}$ if $E'(\calL)=E(\mathcal{M})$. Following the notation in \cite{artin2000semiprime}, we will denote such a
tensor product by $\calL \cdot \mathcal{M}$.

We can now introduce some more geometric objects that will be needed. Let $\mathcal{E}$ be a coherent sheaf of
$\calO_Y$-orders inside $T$, $\tau$ an automorphism of $T$, and $\mathcal{B}_1$ an invertible lattice in $T$ such that
$E(\mathcal{B}_1)=\mathcal{E}$ and $E'(\mathcal{B}_1)=\mathcal{E}^{\tau}$. One can then define a sequence of sheaves
$\{\mathcal{B}_n\}$ by $\mathcal{B}_n=\mathcal{B}_1 \otimes_{\mathcal{E}} \mathcal{B}_1^{\tau}
\otimes_{\mathcal{E}^{\tau}} \ldots  \otimes_{\mathcal{E}^{\tau^{n-1}}} \mathcal{B}_1^{\tau^{n-1}}$. The conditions on
$E(\mathcal{B}_1)$ and $E'(\mathcal{B}_1)$ imply that $\mathcal{B}_n=\mathcal{B}_1 \cdot \mathcal{B}_1^{\tau} \cdot
\ldots \cdot \mathcal{B}_1^{\tau^{n-1}}$. Using this data one can define a \emph{sheaf of bimodule
algebras} $\mathbb{B}(\mathcal{E},\mathcal{B}_1,\tau)=\bigoplus_{n \in \N}
\mathcal{B}_n$ with $\mathcal{B}_0=\mathcal{E}$ and
multiplication given by the formula $\mathcal{B}_i \cdot \mathcal{B}_j^{\tau^{i}} = \mathcal{B}_{i+j}$ for all $i, j
\geq 1$ (see \cite[pg. 102]{artin2000semiprime}).

There is a notion of ampleness for sequences of sheaves such as $\{\mathcal{B}_n\}$, which we now define.
\begin{defn}[{\cite[pg. 99]{artin2000semiprime}}]\label{def: amplebimodule}
Let $\{\mathcal{L}_n\}$ be a sequence of coherent $\mathcal{O}_Y$-modules. The sequence $\{\mathcal{L}_n\}$ is
\emph{ample} if for all coherent sheaves $\mathcal{G}$ of $\mathcal{O}_Y$-modules
and all $n \gg 0$, the sheaf $\mathcal{G} \otimes_{\mathcal{O}_Y} \mathcal{L}_n$ is generated by global sections, and
$H^1(Y,\mathcal{G} \otimes_{\mathcal{O}_Y} \mathcal{L}_n)=0$.
\end{defn}

Suppose now that $\mathcal{B}_1$ is an invertible lattice such that $E(\mathcal{B}_1)=\mathcal{E}$,
$E'(\mathcal{B}_1)=\mathcal{E}^{\tau}$, and the sequence of sheaves $\{\mathcal{B}_n\}$ defined above is ample in the
sense of Definition \ref{def: amplebimodule}. We say that $\mathcal{B}_1$ is an \emph{ample lattice}. One does not need an ample lattice to construct the rings in the next definition, however this condition does
imply that the rings constructed are noetherian. 
\begin{defn}[{\cite[pgs. 103-104]{artin2000semiprime}}]\label{defn: twistedhomring}
Let $Y$, $\tau$, $\mathcal{E}$ and $\{\mathcal{B}_n\}$ be as above. The \emph{generalised twisted homogeneous coordinate
ring} associated to this data is the ring 
\begin{equation*}
B(\mathcal{E},\mathcal{B}_1,\tau)=\bigoplus_{n \in \N} H^0(Y,\mathcal{B}_n)z^n,  
\end{equation*}
whose the multiplication is induced by that in the corresponding sheaf of bimodule algebras, i.e.
$z\beta=\beta^{\tau}z$.
\end{defn}

Such rings were originally studied in \cite{van1996translation}, although Van den Bergh's definitions differ slightly.
For brevity we will refer to such rings as \emph{twisted rings} in future.

Artin and Stafford's main results need several technical hypotheses, however we can summarise them in a simplified
case. 
\begin{theorem}[{\cite[Theorems 0.3 and 0.5, Corollary 0.4]{artin2000semiprime}}]\label{thm: artinstaffordmain}
Suppose that $R$ is a semiprime noetherian c.g.\ algebra of GK dimension 2. Then there is a Veronese subring of $R$
which in high degree is a left ideal in a twisted ring for which $\mathcal{B}_1$ is an ample lattice. Moreover, this
twisted ring is a finite left module over the Veronese ring of $R$. 
\end{theorem}

We will show that $B^{G,\mu}$ fits into the classification in a particularly nice way, in the respect that it is isomorphic to a twisted ring. 

To obtain a geometric description of $B^{G,\mu}$ we first go back to $B$. By \cite[Proposition 2.1]{smith1994center},
this ring embeds in the Ore extension $k(E)[z,z^{-1};\sigma]$, which is its graded quotient ring. Here the function field $k(E)$ is the graded division ring of the homogeneous coordinate ring of the elliptic curve, defined in \eqref{eq:
thcroreextn}. The action of $G$ on $B$ extends to $k(E)[z,z^{-1};\sigma]$ via localisation, inducing a $G$-grading under which $B$ is a graded subring. 

Note that the skew generator $z$ can be replaced with another element in $k(E)z$ up to changing $\sigma$ by a conjugation, which is trivial in $k(E)$. Moreover, $B_1 \subset k(E)z$ and $x_0 \in B_1 \cap B_e$, hence we can assume that $z$ is fixed by the action of $G$ without changing $\sigma$.

The action of $G$ on $k(E)[z,z^{-1};\sigma]$ must preserve its skew structure, hence for all $f \in k(E)$ and $g \in G$ one has 
\begin{equation}\label{eq: sigmagcommute}
(f^g)^{\sigma}=(f^{\sigma})^g. 
\end{equation}

By \cite[Chapter I, Theorem 4.4]{hartshorne1977algebraic} an algebra automorphism on $k(E)$ induces a unique automorphism on $E$ itself. The induced actions of $G$ and $\sigma$ must therefore also commute on the curve by \eqref{eq: sigmagcommute}.
\begin{lemma}\label{lem: actionscoincide}
This action of $G$ on $E$ coincides with that induced by the action on point modules of $A$, given in \eqref{eq:
Gactonpoints}. 
\end{lemma}
\begin{proof}
Both actions are a consequence of the action of $G$ on $B$; one by localisation, the other through the action on ideals
defining points in $\mathbb{P}^3_k$. 
\end{proof}

This lemma allows us to extend the conclusions of Proposition \ref{prop: sklyaninfatpoints} to $\text{qgr}(A^{G,\mu})$.
\begin{corollary}\label{cor: qgrisos}
Let $p, q \in E$. There is an isomorphism $\pi(M_p^2)\cong \pi(M_q^2)$ in $\text{qgr}(A^{G,\mu})$ if and only if $q \in
[p]$.
\end{corollary}
\begin{proof}
By Corollary \ref{cor: fatpointisoclasses} one has $M_p^2 \cong M_{p^{g}}^2$ in $\text{grmod}(A^{G,\mu})$ for all $g \in
G$. Thus suppose that $\pi(M_p^2)\cong \pi(M_q^2)$ for some $p,q \in E$. As in the proof of Proposition \ref{prop: sklyaninfatpoints}, this implies the existence of $n \in \N$ such that $M_{p^{\sigma^{n}}}^2\cong M_{q^{\sigma^{n}}}^2$.
Corollary \ref{cor: fatpointisoclasses} then tells us that $\sigma^n(p)$ and $\sigma^n(q)$ lie in the same $G$-orbit. By
\eqref{eq: sigmagcommute} and Lemma \ref{lem: actionscoincide} the actions of $G$ and $\sigma$ on $E$ commute, therefore
$p$ and $q$ lie in the same $G$-orbit.
\end{proof}

As $B$ is a $G$-graded subring of its graded quotient ring, we can twist both rings simultaneously and use the invariant
ring construction to see that
\begin{equation}\label{eq: twistthcrinvariant}
B^{G,\mu}=M_2(B)^G \hookrightarrow M_2(k(E)[z,z^{-1};\sigma])^G = M_2(k(E))^G[z,z^{-1};\tilde{\tau}],
\end{equation}
where $\tilde{\tau}$ denotes the induced action of $\sigma$. Here we have used the fact that $G$ acts trivially on $z$.

We now state our main theorem. 
\begin{theorem}\label{theorem: geomdescthcrtwist}
There is an isomorphism of $k$-algebras
\begin{equation*}
B^{G,\mu} \cong B(\mathcal{E},\mathcal{B}_1,\tilde{\tau})=\bigoplus_{n \in \N} H^0(E^G,\mathcal{B}_n)z^n \subset
M_2(k(E))^G[z,z^{-1};\tilde{\tau}], 
\end{equation*}
for the following geometric data:
\begin{itemize}
 \item[(i)] an elliptic curve $E^G := E/G$, with $\tau$ the action induced by $\sigma$, along with the morphism of
curves $r:E \rightarrow E^G$ sending $p \mapsto [p]$;
 \item[(ii)] $\mathcal{E}=M_2(r_{\ast}\calO_E)^G$, a sheaf of $\mathcal{O}_{E^{G}}$-orders inside $M_2(k(E))^G$ on which
the automorphism $\tilde{\tau}$ acts. This automorphism restricts to the induced action of $\tau$ on $k(E^G)$;
 \item[(iii)] $\mathcal{B}_1=M_2(r_{\ast}\calL)^G$, an invertible lattice in $M_2(k(E))^G$, where $\calL$ is as defined
after \eqref{eq: 4sklyanincentre}.
\end{itemize}
\end{theorem}

Since the proof of Theorem \ref{theorem: geomdescthcrtwist} is quite technical, we will first prove several preliminary
lemmas. Note that the sheaf $\mathcal{E}$ is defined on an open set $U \subseteq E^G$, with $V=r^{-1}(U)$, by
\begin{equation*}
\mathcal{E}(U)=\left[M_2(r_{\ast}\calO_E)^G\right](U)=M_2((r_{\ast}\calO_E)(U))^G = M_2(\calO_E(V))^G. 
\end{equation*}

Before stating the first lemma needed to prove the main theorem, recall that $\sigma$ is given by translation by a point of infinite order on $E$.
\begin{lemma}\label{lem: ellcurve}
Define $E^G:= E/G$, the orbit space of $E$ under the action of $G$. Then $E^G$ is a smooth elliptic curve, with an
associated automorphism $\tau$ that is induced by $\sigma$. Furthermore, $\tau$ has infinite order and does not fix any
points.
\end{lemma}
\begin{proof}
The orbit space $E^G$ is the curve associated to the fixed field $k(E)^G$, which has transcendence degree 1 over $k$. In
particular we have $k(E^G)=k(E)^G$. As remarked after \eqref{eq: sigmagcommute}, the actions of $G$ and $\sigma$ commute
on $E$. One may therefore conclude that there is an induced action of $\sigma$ on $E^G$. Denote this map by $\tau$,
which is defined by $[p]^{\tau}:=[p^{\sigma}]$ for all $p \in E$, and suppose that $\tau$ has a fixed point $[p]$. Since
$[p]$ contains 4 or fewer points and $\sigma$ is given by translation, this implies that $\sigma$ has finite order which
is a contradiction. A similar argument proves that $\tau$ has infinite order. 

Let us now show that $E^G$ is smooth. The singular locus of $E^G$ must be finite and preserved by $\tau$. If it were
non-empty then this would imply that $\sigma$ has finite order, which is a contradiction. One can now apply Hurwitz's
Theorem \cite[Chapter IV, Corollary 2.4]{hartshorne1977algebraic} to see that $E^G$ has genus 0 or 1; if it had genus 0
then it would be birationally equivalent to \proj{k}{1}, whose automorphisms always fix at least one point. Thus $E^G$
must have genus 1 and hence be an elliptic curve. 
\end{proof}

Before our next lemma we need to define outer and $X$-outer actions of a group. Our statement of the latter property is
given for prime Goldie rings, in which case \cite[Examples 3.6 and 3.7]{montgomery1980fixd} allow us to give the
formulation below.
\begin{defn}\label{def: Xouter}
Let $G$ be a finite group acting by ring automorphisms on a ring $R$. We say that $G$ is \emph{outer} if no nontrivial element of $G$ acts by conjugation by an element of $R$. If, in addition, we assume that $R$ is
prime Goldie, we say that $G$ is \emph{X-outer} on $R$ if it is outer when extended to the
Goldie quotient ring $Q(R)$.
\end{defn}
\begin{lemma}\label{lem: Xouter}
The group $G$ is outer on $M_2(k(E))$ and therefore X-outer on the ring $M_2(\calO_E(V))$ for each $G$-invariant affine
open set $V \subseteq E$. Consequently, $M_2(\calO_E(V))^G$ is prime.
\end{lemma}
\begin{proof}
Consider the action of $G$ on $M_2(k)$ given in \cite[Equation (1.4)]{davies2014cocycle1}. As the action of $G$ on $B_1$ affords the regular representation, each graded component of $k(E)$ under the induced $G$-grading is non-empty. Hence we can find a non-zero element $y \in k(E)_{g_{2}}$, in which case:
\begin{equation*}
\begin{pmatrix} y & 0 \\ 0 & 0 \end{pmatrix}^{g_{1}}=\begin{pmatrix} -y & 0 \\ 0 & 0 \end{pmatrix}.
\end{equation*}

Suppose that this action were given by conjugation by $\left( \begin{smallmatrix} a&b\\ c&d \end{smallmatrix} \right)
\in M_2(k(E))$:
\begin{gather}
\begin{aligned}\label{eq: conjg1calc}
\begin{pmatrix} -y & 0 \\ 0 & 0 \end{pmatrix} 
&= \frac{1}{ad-bc} \begin{pmatrix} a & b \\ c & d \end{pmatrix} \begin{pmatrix} y & 0 \\ 0 & 0 \end{pmatrix}
\begin{pmatrix} d & -b \\ -c & a \end{pmatrix}  \\
&= \frac{1}{ad-bc} \begin{pmatrix} ad y & -ab y \\ c d y & -bc y \end{pmatrix}.
\end{aligned}
\end{gather}
In order for the correct entries of the matrix to vanish we would need $b=c=0$. But then \eqref{eq: conjg1calc} would contradict $y \neq 0$.

We must repeat these calculations for the other two non-identity elements in the group. For $g=g_2$ or $g_1g_2$, choose
a non-zero element $y' \in k(E)_{g_{1}}$. In both cases we have
\begin{equation*}
\begin{pmatrix} y' & 0 \\ 0 & 0 \end{pmatrix}^{g}= \begin{pmatrix} 0 & 0 \\ 0 & -y' \end{pmatrix}.
\end{equation*}
A similar argument to that above shows that neither element can act by conjugation by an element in $M_2(k(E))$.

Now consider a $G$-invariant affine open set $V \subseteq E$ and $M_2(\calO_E(V))$. As $\calO_E(V)$ is a domain it is
clear that $M_2(\calO_E(V))$ is prime. Furthermore, it has Goldie quotient ring $M_2(k(E))$, on which the action of $G$
is outer by the argument above. Thus the action of $G$ on $M_2(\calO_E(V))$ is X-outer, whereupon we can apply \cite[Theorem 3.17(2)]{montgomery1980fixd} to show that $M_2(\calO_E(V))^G$ is a prime ring.
\end{proof}

We now turn to the sheaf of orders which will be used to describe $B^{G,\mu}$. 
\begin{lemma}\label{lem: maxorder}
Define $\mathcal{E}:=M_2(r_{\ast}\calO_E)^G$, which is considered as a subsheaf of the constant sheaf $M_2(k(E))^G$.
Then
\begin{itemize}
\item[(i)] the natural action of $\sigma$ on $M_2(k(E))$ restricts to an automorphism $\tilde{\tau}$ on $\mathcal{E}$,
with $\tilde{\tau}$ restricting to the induced action of $\tau$ on $k(E^G)$;
\item[(ii)] $\mathcal{E}$ is a sheaf of Dedekind prime rings and therefore a sheaf of maximal orders.
\end{itemize}
\end{lemma}
\begin{proof}
The automorphism $\sigma$ extends naturally from $k(E)$ to $M_2(k(E))$ under the trivial action on the matrix units.
Since this automorphism commutes with $G$, it restricts to an automorphism $\tilde{\tau}$ on $\mathcal{E}$. It follows
from the construction that the action of $\tilde{\tau}$ on $k(E)^G$ coincides with that of $\tau$.

To prove (ii), let $U \subseteq E^G$ be an affine open set with $V=r^{-1}(U)$. As $V \subseteq E$ is $G$-invariant, $\mathcal{E}(U)= M_2(\calO_E(V))^G$ is prime by Lemma \ref{lem: Xouter}. Since $G$ acts on
$\calO_E(V)$ by $k$-algebra automorphisms, one can view $\mathcal{E}(U)$ as a cocycle twist of the form
$\calO_E(V)^{G,\mu}$, where $G$ and $\mu$ are the same group and 2-cocycle as used to twist $B$ respectively. Note that $\calO_E(V)$ is noetherian and has global
dimension 1 since $E$ is a smooth curve, therefore we may apply \cite[Corollary 4.11 and Proposition 4.14]{davies2014cocycle1} to see that $\mathcal{E}(U)$ is a hereditary noetherian prime (HNP) ring.

We will now show that $\mathcal{E}(U)$ contains no idempotent maximal ideals, which by \cite[Proposition
5.6.10]{mcconnell2001noncommutative} will imply that it contains no nontrivial idempotent ideals. The existence of
idempotent maximal ideals is a local condition: given a sheaf $\mathcal{M}$ of maximal ideals of $\mathcal{E}$, the
sheaf $\mathcal{E}/\mathcal{M}$ is supported at a single point $p \in E^G$, and the stalk $\mathcal{M}_p$ is an
idempotent maximal ideal of $\mathcal{E}_p$ if and only if $\mathcal{M}(V)$ is an idempotent maximal ideal of
$\mathcal{E}(V)$ for an affine open set $V \ni p$.

Let $\bigcup_{i=1}^n U_i$ be an affine open cover of $E^G$ for some $n \in \N$, which gives rise to a corresponding
cover $\bigcup_{i=1}^n V_i$ of $E$ by $G$-invariant open sets with $V_i = r^{-1}(U_i)$. As the rings $\mathcal{E}(U_i)
\subseteq M_2(k(E))$ are PI and noetherian, they can have at most finitely many idempotent maximal ideals by \cite[Theorem
13.7.15]{mcconnell2001noncommutative}. By the remarks of the preceding paragraph, there can be at most finitely many
idempotent maximal ideals in the totality of the stalks of $\mathcal{E}$. We know that $\tau$ has no finite orbits on
$E^G$. This means that if a stalk $\mathcal{E}_p$ contains an idempotent maximal ideal $\mathcal{M}_p$, then there are
infinitely many other stalks with such an ideal, of the form $\tilde{\tau}^n(\mathcal{M}_p)$ for $n \in \N$. This is a
contradiction.

Now we may apply \cite[Theorem 5.6.3]{mcconnell2001noncommutative} to $\mathcal{E}(U)$: it contains no nontrivial
idempotent ideals and therefore it is a Dedekind prime ring. In particular, condition (ii) of Theorem 5.2.10 op. cit.
implies that $\mathcal{E}(U)$ is a maximal order.
\end{proof}

\begin{lemma}\label{lem: twistedring}
Define $\mathcal{B}_1:=M_2(r_{\ast}\calL)^G$ which, like $\mathcal{E}$, is considered as a subsheaf of the constant
sheaf $M_2(k(E))^G$. Then 
\begin{itemize}
 \item[(i)] $\mathcal{B}_1$ is an invertible $\calO_{E^G}$-lattice;
 \item[(ii)]  $\mathcal{E}$ is coherent sheaf of maximal $\calO_{E^G}$-orders for which
$\mathcal{E}=\mathcal{E}^{\tilde{\tau}}$;
 \item[(iii)] $\mathcal{E}= E(\mathcal{B}_1) = E'(\mathcal{B}_1)$.
\end{itemize}
\end{lemma}
\begin{proof}
We first show that $\mathcal{B}_1$ is an $\calO_{E^G}$-lattice. To do this we need to prove that $M_2(k(E))^G =
\mathcal{B}_1 \cdot k(E^G)$, where it suffices to work at the level of the global sections of $\mathcal{B}_1$. Observe
that the embedding in \eqref{eq: twistthcrinvariant} sends $B^{G,\mu}_1$ to $M_2(\text{H}^0(E,\calL))^Gz$, which is
precisely $\text{H}^0(E^G,\mathcal{B}_1)z$ since
\begin{equation*}
\text{H}^0(E^G,\mathcal{B}_1) =  \left[M_2(r_{\ast}\calL)^G\right](E^G) = M_2(\calL(E))^G=M_2(\text{H}^0(E,\calL))^G.
\end{equation*} 
Thus $\text{H}^0(E^G,\mathcal{B}_1)$ contains the elements
\begin{equation}\label{eq: sorted}
\begin{pmatrix} x_0 & 0 \\ 0 & x_0 \end{pmatrix}z^{-1},\begin{pmatrix} x_{1} & 0 \\ 0 &
 -x_{1} \end{pmatrix}z^{-1},\begin{pmatrix} 0 & x_{2} \\ x_{2} & 0 \end{pmatrix}z^{-1},\begin{pmatrix} 0 & -x_3 \\ x_{3}
& 0 \end{pmatrix}z^{-1},
\end{equation}
where $x_0,x_1,x_2$ and $x_3$ are the degree 1 generators of $B$. The elements in \eqref{eq: sorted} are linearly
independent over $k(E^G)$. But $M_2(k(E))^G$ is a 4-dimensional vector space over $k(E^G)$, which implies that
$M_2(k(E))^G = \mathcal{B}_1 \cdot k(E^G)$ must hold and therefore $\mathcal{B}_1$ is an $\calO_{E^G}$-lattice.

Notice that $\mathcal{B}_1$ is an $\mathcal{E}$-module on the right since
\begin{equation*}
M_2(r_*\mathcal{L})^G\cdot M_2(r_*\mathcal{O}_E)^G \subseteq 
 M_2(r_*\mathcal{L}\cdot r_*\mathcal{O}_E)^G \subseteq M_2(r_*\mathcal{L})^G. 
\end{equation*}
In particular, one has $\mathcal{E} \subseteq E'(\mathcal{B}_1)$. One may argue in a similar manner for left modules to
see that $\mathcal{E} \subseteq E(\mathcal{B}_1)$. Applying \cite[Lemma 3.1.12(i)]{mcconnell2001noncommutative} shows
that the orders $E(\mathcal{B}_1)$ and $E'(\mathcal{B}_1)$ are equivalent to $\mathcal{E}$. However, we showed in Lemma
\ref{lem: maxorder} that $\mathcal{E}$ is a sheaf of maximal orders, therefore we must have $\mathcal{E}=
E(\mathcal{B}_1) = E'(\mathcal{B}_1)$. By \cite[Lemma 1.10(i)]{artin2000semiprime}, $\mathcal{E}$ must be a coherent
sheaf of $\calO_{E^G}$-modules. The equality $\mathcal{E}=\mathcal{E}^{\tilde{\tau}}$ follows from Lemma \ref{lem:
maxorder}(i). 

Finally, we show that $\mathcal{B}_1$ is invertible over $\mathcal{E}$. Note that $\mathcal{B}_1$ is generated by its global sections (since $\calL$ is), therefore $\mathcal{B}_1$ is a sheaf of f.g.\ modules over $\mathcal{E}$. Since
$\mathcal{E}$ is a sheaf of Dedekind prime rings by Lemma \ref{lem: maxorder}, \cite[Lemma 5.7.4]{mcconnell2001noncommutative} implies that $\mathcal{B}_1 \subseteq M_2(k(E))^G$ must be a sheaf of torsionfree, projective modules contained in the Goldie quotient ring of $\mathcal{E}$. These facts imply that it must be an invertible lattice.
\end{proof}

We are now in a position to prove Theorem \ref{theorem: geomdescthcrtwist}.
\begin{proofof}{Theorem \ref{theorem: geomdescthcrtwist}}
To begin, note that in Lemma \ref{lem: twistedring} we proved that $\mathcal{E}= E(\mathcal{B}_1) = E'(\mathcal{B}_1)$,
whereupon
\begin{equation}\label{eq: productsheaf}
\mathcal{B}_n = \mathcal{B}_1 \cdot \mathcal{B}_1^{\tilde{\tau}} \cdot \ldots \cdot \mathcal{B}_1^{\tilde{\tau}^{n-1}},
\end{equation}
for all $n \in \N$. As all sheaves involved lie inside $M_2(k(E))^G$, we may construct the twisted ring
$B(\mathcal{E},\mathcal{B}_1,\tilde{\tau})$ as a subring of the Ore extension $M_2(k(E))^G[z,z^{-1};\tilde{\tau}]$, as
described in the statement of the theorem.

The degree 1 piece of $B(\mathcal{E},\mathcal{B}_1,\tilde{\tau})$ is $\text{H}^0(E^G,\mathcal{B}_1)z =
M_2(\text{H}^0(E,\calL))^Gz$. Recall that in the proof of Lemma \ref{lem: twistedring} we observed that this is
precisely the image of $B_1^{G,\mu}$ under its embedding in the Ore extension. By Lemma \ref{lem: twistthcr} $B^{G,\mu}$ is generated in degree 1, therefore we obtain an embedding of $k$-algebras $B^{G,\mu} \hookrightarrow
B(\mathcal{E},\mathcal{B}_1,\tau)$. In particular, there is an injection of vector spaces $B_n^{G,\mu} \hookrightarrow
H_0(E^G,\mathcal{B}_n)z^n$ for all $n \in \N$. 

Let us now consider an open set $U \subseteq E^G$, with $V=r^{-1}(U)$. Since $\tau$ is induced by $\sigma$, we have
$r^{-1}(\tau^n(U)) = \sigma^n(V)$ for all $n \in \N$. Note that $\tilde{\tau}$ commutes with the action of $G$ on
$M_2(k(E))^G$, which is a consequence of $\sigma$ commuting with the group action. Observe finally that as $\mathcal{E}$
is a sheaf of maximal orders and $\mathcal{B}_1$ is an invertible lattice contained in $M_2(k(E))^G$, the tensor
products appearing in the definition of $\mathcal{B}_n$ are actual products inside $M_2(k(E))^G$. Thus
\begin{align*}
\mathcal{B}_n(U) &= \mathcal{B}_1(U)\cdot \mathcal{B}_1^{\tilde{\tau}}(U) \cdot  \ldots \cdot
\mathcal{B}_1^{\tilde{\tau}^{n-1}}(U)\\
&=  \left[M_2(r_{\ast}\calL)^G\right](U) \cdot \left[M_2((r_{\ast}\calL)^{\tau})^G\right](U) \cdot \ldots \cdot 
\left[M_2((r_{\ast}\calL)^{\tau^{n-1}})^G\right](U) \\
&= M_2(\calL(V))^G \cdot M_2(\calL(\sigma(V)))^G \cdot \ldots \cdot M_2(\calL(\sigma^{n-1}(V)))^G\\
&\subseteq M_2(\calL(V)\calL(\sigma(V)) \ldots \calL(\sigma^{n-1}(V)))^G\\
&= M_2(\calL\calL^{\sigma} \ldots \calL^{\sigma^{n-1}}(V))^G\\
&= \left[M_2(r_{\ast}\calL_n)^G\right](U).
\end{align*}

We therefore have an inclusion $\mathcal{B}_n \hookrightarrow M_2(r_{\ast}\calL_n)^G$, and by taking global sections one
has $\text{H}^0(E^G,\mathcal{B}_n) \subseteq  M_2(H^0(E,\calL_n))^G$. But since the latter vector space is precisely
$B^{G,\mu}_n$, this implies that the injection $B_n^{G,\mu} \hookrightarrow H_0(E^G,\mathcal{B}_n)z^n$ must be a
bijection, which completes the proof.
\end{proofof}
\begin{rem}\label{rem: ample}
Although we did not show during the proof of Theorem \ref{theorem: geomdescthcrtwist} that the sequence $\{\mathcal{B}_n\}$
is ample, it follows from that result as we now indicate. Since $B^{G,\mu}$ is semiprime noetherian of GK dimension 2,
the sheaves associated to its graded components form an ample sequence \cite[Proposition 6.4]{artin2000semiprime};
Theorem \ref{theorem: geomdescthcrtwist} shows that these sheaves coincide with the sequence $\{\mathcal{B}_n\}$.
\end{rem}

\begin{corollary}\label{cor: astfequivcat}
There is an equivalence of categories $\text{qgr}(B^{G,\mu}) \simeq \text{coh}(\mathcal{E})$, where
$\text{coh}(\mathcal{E})$ denotes the category of coherent sheaves over
$\mathcal{E}$.
\end{corollary}
\begin{proof}
Combine Theorem \ref{theorem: geomdescthcrtwist} with \cite[Corollary 6.11]{artin2000semiprime}.
\end{proof}

The geometric description of $B^{G,\mu}$ provided by Theorem \ref{thm: artinstaffordmain} allows us to explain and describe many of the properties of that ring. We begin by strengthening Proposition \ref{prop: semiprime}.
\begin{corollary}\label{cor: actuallyprime}
$B^{G,\mu}$ is prime.
\end{corollary}
\begin{proof}
By Theorem \ref{theorem: geomdescthcrtwist} we know that the degree 0 component of the graded quotient ring of $B^{G,\mu}$
is $M_2(k(E))^G$. If we can show that this ring is simple artinian then $B^{G,\mu}$ must be prime. By Lemma \ref{lem:
Xouter} the action of $G$ on $M_2(k(E))$ is outer. As $M_2(k(E))$ is simple artinian, we may apply \cite[Theorem
2.7(1)]{montgomery1980fixd} to conclude that $M_2(k(E))^G$ is also simple artinian, which completes the proof.
\end{proof}

One can explicitly describe the structure of $M_2(k(E))^G$, as the following result shows. In the proof we use the
notion of \emph{PI degree} defined in \cite[\S 13.6.7]{mcconnell2001noncommutative}.
\begin{corollary}\label{cor: shapeofsimpleart}
There is an isomorphism of $k$-algebras $M_2(k(E))^G \cong M_2(k(E^G))$, where $k(E^G)=k(E)^G$.
\end{corollary}
\begin{proof}
By Corollary \ref{cor: actuallyprime}, $M_2(k(E))^G$ is simple artinian and therefore isomorphic to a ring of the form
$M_n(D)$ for some division ring $D$. Considering the PI degree of $M_2(k(E))^G$ enables one
to conclude that either $n=2$ and $D$ is a field or $M_2(k(E))^G$ is a division ring that is 4-dimensional over its
centre. We can rule out the latter case since $M_2(k(E))^G$ is the degree 0 component of the graded quotient ring of
$B^{G,\mu}$, which contains non-regular homogeneous elements by Remark \ref{rem: bgmunilp}. Thus $D$ must be a field,
and therefore the centre of $M_2(D)$ is isomorphic to $D$ itself. Since $k(E^G)$ is central in $M_2(k(E))^G$, it must
embed in $D$ under the isomorphism $M_2(k(E))^G \cong M_2(D)$. 

To complete the proof, observe that $M_2(k(E))$ is a 4-dimensional module over $M_2(k(E))^G$ (on either the left or the
right), which itself is 4-dimensional over its centre $Z(M_2(k(E))^G) \cong D$. However, it is clear that $M_2(k(E))$ is
16-dimensional over $k(E^G)$, in which case one must have $D=k(E^G)$. 
\end{proof}

Our work in the remainder of this section concerns the irreducible objects in $\text{qgr}(B^{G,\mu})$. We begin with a lemma.
\begin{lemma}\label{lem: irrtail1crit}
The irreducible objects in $\text{qgr}(B^{G,\mu})$ are precisely the tails of 1-critical $B^{G,\mu}$-modules. 
\end{lemma}
\begin{proof}
Let us consider an irreducible object in $\text{qgr}(B^{G,\mu})$. Such an object will have the form $\pi(M)$ for some $M
\in \text{grmod}(B^{G,\mu})$. Theorem \ref{thm: thcrtwistprops}(ii) implies that $B^{G,\mu}$ has GK dimension 2, in
which case \cite[Proposition 5.1(d)]{krause2000growth} tells us that $\text{GKdim }M=0,1$ or 2. By Corollary \ref{cor: astfequivcat} the dimension must be 1, since in the proof of that result the images of skyscraper sheaves from $\text{coh}(\mathcal{E})$ will have dimension 1 by construction. 
\end{proof}

We can now give a more concrete description of the irreducible objects in $\text{qgr}(B^{G,\mu})$. 
\begin{proposition}\label{prop: onlyfatpoints}
Any irreducible object in $\text{qgr}(B^{G,\mu})$ is isomorphic to $\pi(M_p^2)$ for some $p \in E$, where $M_p^2$ is a
direct sum of point modules over $B$. In particular, such modules are 1-critical and are therefore fat point modules of
multiplicity 2.
\end{proposition}
\begin{proof}
By Lemma \ref{lem: irrtail1crit}, any irreducible object in $\text{qgr}(B^{G,\mu})$ has the form $\pi(N)$ for some
1-critical $B^{G,\mu}$-module $N$. Consider the extension $N \otimes_{B^{G,\mu}} M_{2}(B)$ as a right
$M_2(B)$-module, which is $\N$-graded via the standard grading on a tensor product of graded modules. One may argue in
the same manner as in the proof of \cite[Proposition 4.32(i)]{davies2014cocycle1} to show that $N \otimes_{B^{G,\mu}} M_{2}(B)$ has
GK dimension 1. This $M_{2}(B)$-module is noetherian, hence there is a 1-critical $M_{2}(B)$-module, $I$ say, that
embeds in $\left(N \otimes_{B^{G,\mu}} M_{2}(B)\right)_{\geq n}$ for some sufficiently large $n \in \N$. One therefore
has $\pi(I) \hookrightarrow \pi(N \otimes_{B^{G,\mu}} M_{2}(B))$ in $\text{qgr}(M_2(B))$. The 1-critical modules over
$B$ are isomorphic in high degree to point modules by the equivalence of categories $\text{qgr}(B) \simeq
\text{coh}(\calO_E)$ from \cite[Theorem 1.3]{artin1990twisted}. Thus we may replace $\pi(I)$ with $\pi(M_p^2)$ for some
$p \in E$.

Let us now restrict from $\text{qgr}(M_2(B))$ to $\text{qgr}(B^{G,\mu})$. By \cite[Lemma 4.6]{davies2014cocycle1} one has 
\begin{gather}
\begin{aligned}\label{eq: restinject}
\pi(M_p^2) \hookrightarrow \pi(N \otimes_{B^{G,\mu}} M_2(B)) &\cong \pi\left(N \otimes_{B^{G,\mu}} \left(\bigoplus_{g \in
G} {^{\text{id}}}(B^{G,\mu})^{\phi_g} \right)\right) \\
&\cong \pi\left(\bigoplus_{g \in G} N^{\phi_g}\right). 
\end{aligned}
\end{gather}

The summands $N^{\phi_g}$ are 1-critical $B^{G,\mu}$-modules, and moreover they are the factors in a critical
composition series for $(N \otimes_{B^{G,\mu}} M_2(B))_{B^{G,\mu}}$. By \cite[Proposition 1.5]{smith1992the}, these
factors are unique up to permutation and isomorphism in high degree. However, we claim that the right $M_2(B)$-module
$M_p^2$ remains 1-critical upon restriction to a module over $B^{G,\mu}$, in which case $\pi(M_p^2) \cong
\pi(N^{\phi_{g}})$ for some $g \in G$. To prove this claim, observe that the module $M_p$ is also a point module over
$A$, thus $M_p^2$ can be considered as an $M_2(A)$-module that has been restricted first to $A^{G,\mu}$ and then to the
factor ring $B^{G,\mu}$. Thus Proposition \ref{prop: fatpoints} is applicable and implies that $M_p^2$ is a 1-critical
$B^{G,\mu}$-module. Since such a module has Hilbert series $2/(1-t)$, it must be a fat point module of multiplicity 2.

By untwisting one has $\pi((M_p^2)^{\phi_g^{-1}}) \cong \pi(N)$ in $\text{qgr}(B^{G,\mu})$. Note that the $G$-graded
automorphism $\phi_g$ extends naturally from $B^{G,\mu}$ to $M_2(B)$ via the $G$-grading on $B$. Thus we may consider
$(M_p^2)^{\phi_g^{-1}}$ as a right $M_2(B)$-module, in which case there exists $L \in \text{grmod}(B)$ such that
$(M_p^2)^{\phi_g^{-1}} \cong L^2$ in $\text{grmod}(B)$. By considering Hilbert series it is clear that $L$ must be a
point module. On restriction to $B^{G,\mu}$, one obtains an isomorphism $\pi(M_q^2) \cong \pi(N)$ in
$\text{qgr}(B^{G,\mu})$ for some $q \in E$. This completes the proof. 
\end{proof}

Proposition \ref{prop: onlyfatpoints} enables us to describe the geometric origins of the fat point modules of multiplicity 2 that we have studied.
\begin{proposition}\label{prop: fatpointsincohE}
For $p \in E$ consider the coherent sheaf of right $M_2(\calO_E)$-modules $k_p^2$, where
$k_p:=\calO_E/\mathcal{I}_p$ is a skyscraper sheaf supported at
$p$. Under a natural functor $\phi: \text{coh}(M_2(\mathcal{O}_E)) \rightarrow \text{coh}(\mathcal{E})$, the irreducible
objects in $\text{coh}(\mathcal{E})$ are of the form $\phi(k_p^2)$. Furthermore, $\phi(k_p^2) \cong \phi(k_q^2)$ if and
only if $q \in [p]$.
\end{proposition}
\begin{proof}
By \cite[Corollary 6.11]{artin2000semiprime} there is an equivalence of categories 
$$\varphi_1: \text{coh}(M_2(\mathcal{O}_E)) \rightarrow \text{qgr}(M_2(B)).$$
We label the equivalence from Corollary \ref{cor: astfequivcat} by $\varphi_3: \text{qgr}(B^{G,\mu}) \rightarrow \text{coh}(\mathcal{E})$. Finally, let $\varphi_2$ denote
the functor from $\text{qgr}(M_2(B))$ to $\text{qgr}(B^{G,\mu})$ obtained by the restriction of modules. Consider the
diagram below:
\begin{equation}\label{eq: cohqgrdiag}
\centering
\begin{tikzpicture}
\draw [->] (1.25,0) -- (1.85,0); 
\node[below] at (1.55,0) {\tiny{$\sim$}};
\node[above] at (1.55,0) {\footnotesize{$\varphi_1$}};
\draw [dashed,->] (0,-0.25) -- (0,-1.25); 
\node[left] at (0,-0.75) {\footnotesize{$\phi$}};
\draw [->] (3,-0.25) -- (3,-1.25); 
\node[right] at (3,-0.75) {\footnotesize{$\varphi_2$}};
\draw [<-] (0.65,-1.5) -- (2,-1.5);
\node[above] at (1.325,-1.5) {\tiny{$\sim$}}; 
\node[below] at (1.325,-1.5) {\footnotesize{$\varphi_3$}};
\node at (0,0) {$\text{coh}(M_2(\mathcal{O}_E))$};
\node at (0,-1.5) {$\text{coh}(\mathcal{E})$};
\node at (3,0) {$\text{qgr}(M_2(B))$};
\node at (3,-1.5) {$\text{qgr}(B^{G,\mu})$};
\end{tikzpicture}
\end{equation} 
We denote the clockwise composition of functors beginning at $\text{coh}(M_2(\mathcal{O}_E))$ by $\phi$. 

Let us consider an irreducible object in $\calF \in \text{coh}(\mathcal{E})$. By Proposition \ref{prop:
onlyfatpoints} we know that $\varphi_3(\varphi_2(\pi(M_p^2)))=\calF$ for some $p \in E$. One also has
$\varphi_1(k^2_p)=\pi(M_p^2)$, therefore $\phi(k_p^2)=\calF$. Since $\varphi_1$ and $\varphi_3$ are equivalences, the
isomorphisms among the sheaves $\phi(k_p^2)$  are governed by the restriction functor $\varphi_2$. Corollary \ref{cor:
qgrisos} describes the only such isomorphisms that can occur, and this completes the proof.
\end{proof}

\bibliographystyle{abbrv}
\bibliography{allrefs}

\def\cprime{$'$} \def\cprime{$'$} \def\cprime{$'$} \def\cprime{$'$}
  \def\cprime{$'$} \def\cprime{$'$} \def\cprime{$'$} \def\cprime{$'$}
  \def\cprime{$'$} \def\cprime{$'$}
\begin{thebibliography}{10}

\bibitem{artin1990geometry}
M.~Artin.
\newblock Geometry of quantum planes.
\newblock In {\em Azumaya algebras, actions, and modules ({B}loomington, {IN},
  1990)}, volume 124 of {\em Contemp. Math.}, pages 1--15. Amer. Math. Soc.,
  Providence, RI, 1992.

\bibitem{artin1987graded}
M.~Artin and W.~F. Schelter.
\newblock Graded algebras of global dimension {$3$}.
\newblock {\em Adv. in Math.}, 66(2):171--216, 1987.

\bibitem{artin1999generic}
M.~Artin, L.~W. Small, and J.~J. Zhang.
\newblock Generic flatness for strongly {N}oetherian algebras.
\newblock {\em J. Algebra}, 221(2):579--610, 1999.

\bibitem{artin1995noncommutative}
M.~Artin and J.~T. Stafford.
\newblock Noncommutative graded domains with quadratic growth.
\newblock {\em Invent. Math.}, 122(2):231--276, 1995.

\bibitem{artin2000semiprime}
M.~Artin and J.~T. Stafford.
\newblock Semiprime graded algebras of dimension two.
\newblock {\em J. Algebra}, 227(1):68--123, 2000.

\bibitem{artin1991modules}
M.~Artin, J.~Tate, and M.~Van~den Bergh.
\newblock Modules over regular algebras of dimension {$3$}.
\newblock {\em Invent. Math.}, 106(2):335--388, 1991.

\bibitem{artin1990twisted}
M.~Artin and M.~Van~den Bergh.
\newblock Twisted homogeneous coordinate rings.
\newblock {\em J. Algebra}, 133(2):249--271, 1990.

\bibitem{bazlov2012cocycle}
Y.~Bazlov and A.~Berenstein.
\newblock Cocycle twists and extensions of braided doubles.
\newblock {\em Noncommutative Birational Geometry, Representations and
  Combinatorics, in: Contemp. Math}, 592:19--70, 2012.

\bibitem{cassidy2010generlizations}
T.~Cassidy and M.~Vancliff.
\newblock Generalizations of graded {C}lifford algebras and of complete
  intersections.
\newblock {\em J. Lond. Math. Soc. (2)}, 81(1):91--112, 2010.

\bibitem{chirvasitu2015exotic}
A.~Chirvasitu and S.~P. Smith.
\newblock Exotic elliptic algebras.
\newblock {\em arXiv preprint arXiv:1502.01744}, 2015.

\bibitem{davies2014cocycle1}
A.~Davies.
\newblock Cocycle twists of algebras.
\newblock {\em arXiv preprint, arXiv:1504.06299}, 2014.

\bibitem{davies2014thesis}
A.~Davies.
\newblock Cocycle twists of algebras, Ph.D. Thesis, University of Manchester,
  July 22nd 2014,
  \url{https://www.escholar.manchester.ac.uk/api/datastream?publicationPid=uk-ac-man-scw:229719&datastreamId=FULL-TEXT.PDF}.

\bibitem{hartshorne1977algebraic}
R.~Hartshorne.
\newblock {\em Algebraic geometry}.
\newblock Springer-Verlag, New York, 1977.
\newblock Graduate Texts in Mathematics, No. 52.

\bibitem{krause2000growth}
G.~R. Krause and T.~H. Lenagan.
\newblock {\em Growth of algebras and {G}elfand-{K}irillov dimension},
  volume~22 of {\em Graduate Studies in Mathematics}.
\newblock American Mathematical Society, Providence, RI, revised edition, 2000.

\bibitem{lebruyn1995central}
L.~Le~Bruyn.
\newblock Central singularities of quantum spaces.
\newblock {\em J. Algebra}, 177(1):142--153, 1995.

\bibitem{levasseur1992some}
T.~Levasseur.
\newblock Some properties of noncommutative regular graded rings.
\newblock {\em Glasgow Math. J.}, 34(3):277--300, 1992.

\bibitem{levasseur1993modules}
T.~Levasseur and S.~P. Smith.
\newblock Modules over the {$4$}-dimensional {S}klyanin algebra.
\newblock {\em Bull. Soc. Math. France}, 121(1):35--90, 1993.

\bibitem{mcconnell2001noncommutative}
J.~C. McConnell and J.~C. Robson.
\newblock {\em Noncommutative {N}oetherian rings}.
\newblock Pure and Applied Mathematics (New York). John Wiley \& Sons Ltd.,
  Chichester, 1987.
\newblock With the cooperation of L. W. Small, A Wiley-Interscience
  Publication.

\bibitem{montgomery1980fixd}
S.~Montgomery.
\newblock {\em Fixed rings of finite automorphism groups of associative rings},
  volume 818 of {\em Lecture Notes in Mathematics}.
\newblock Springer, Berlin, 1980.

\bibitem{odesskii2002elliptic}
A.~V. Odesski{\u\i}.
\newblock Elliptic algebras.
\newblock {\em Uspekhi Mat. Nauk}, 57(6(348)):87--122, 2002.

\bibitem{pym2015quantum}
B.~Pym.
\newblock Quantum deformations of projective three-space.
\newblock {\em Adv. Math.}, 281:1216–--1241, 2015.

\bibitem{rogalski2006proj}
Z.~Reichstein, D.~Rogalski, and J.~J. Zhang.
\newblock Projectively simple rings.
\newblock {\em Adv. Math.}, 203(2):365--407, 2006.

\bibitem{rogalski2008canonical}
D.~Rogalski and J.~J. Zhang.
\newblock Canonical maps to twisted rings.
\newblock {\em Math. Z.}, 259(2):433--455, 2008.

\bibitem{shelton2001koszul}
B.~Shelton and C.~Tingey.
\newblock On {K}oszul algebras and a new construction of {A}rtin-{S}chelter
  regular algebras.
\newblock {\em J. Algebra}, 241(2):789--798, 2001.

\bibitem{shelton1999embedding}
B.~Shelton and M.~Vancliff.
\newblock Embedding a quantum rank three quadric in a quantum {$\mathbb{P}^3$}.
\newblock {\em Comm. Algebra}, 27(6):2877--2904, 1999.

\bibitem{shelton1999some}
B.~Shelton and M.~Vancliff.
\newblock Some quantum {$\mathbb{P}^3$}'s with one point.
\newblock {\em Comm. Algebra}, 27(3):1429--1443, 1999.

\bibitem{smith1992the}
S.~P. Smith.
\newblock The 4-dimensional {S}klyanin algebra at points of finite order.
\newblock {\em preprint}, 1992.

\bibitem{smith1994four}
S.~P. Smith.
\newblock The four-dimensional {S}klyanin algebras.
\newblock In {\em Proceedings of {C}onference on {A}lgebraic {G}eometry and
  {R}ing {T}heory in honor of {M}ichael {A}rtin, {P}art {I} ({A}ntwerp, 1992)},
  volume~8, pages 65--80, 1994.

\bibitem{smith1992regularity}
S.~P. Smith and J.~T. Stafford.
\newblock Regularity of the four-dimensional {S}klyanin algebra.
\newblock {\em Compositio Math.}, 83(3):259--289, 1992.

\bibitem{smith1993irreducible}
S.~P. Smith and J.~M. Staniszkis.
\newblock Irreducible representations of the {$4$}-dimensional {S}klyanin
  algebra at points of infinite order.
\newblock {\em J. Algebra}, 160(1):57--86, 1993.

\bibitem{smith1994center}
S.~P. Smith and J.~Tate.
\newblock The center of the {$3$}-dimensional and {$4$}-dimensional {S}klyanin
  algebras.
\newblock In {\em Proceedings of {C}onference on {A}lgebraic {G}eometry and
  {R}ing {T}heory in honor of {M}ichael {A}rtin, {P}art {I} ({A}ntwerp, 1992)},
  volume~8, pages 19--63, 1994.

\bibitem{stafford1994regularity}
J.~T. Stafford.
\newblock Regularity of algebras related to the {S}klyanin algebra.
\newblock {\em Trans. Amer. Math. Soc.}, 341(2):895--916, 1994.

\bibitem{stafford2001noncommutative}
J.~T. Stafford and M.~Van~den Bergh.
\newblock Noncommutative curves and noncommutative surfaces.
\newblock {\em Bull. Amer. Math. Soc. (N.S.)}, 38(2):171--216, 2001.

\bibitem{tate1996homological}
J.~Tate and M.~Van~den Bergh.
\newblock Homological properties of {S}klyanin algebras.
\newblock {\em Invent. Math.}, 124(1-3):619--647, 1996.

\bibitem{van1996translation}
M.~Van~den Bergh.
\newblock A translation principle for the four-dimensional {S}klyanin algebras.
\newblock {\em J. Algebra}, 184(2):435--490, 1996.

\bibitem{vancliff1998some}
M.~Vancliff, K.~Van~Rompay, and L.~Willaert.
\newblock Some quantum {${\mathbb P}^3$}'s with finitely many points.
\newblock {\em Comm. Algebra}, 26(4):1193--1208, 1998.

\bibitem{zhang1998twisted}
J.~J. Zhang.
\newblock Twisted graded algebras and equivalences of graded categories.
\newblock {\em Proc. London Math. Soc. (3)}, 72(2):281--311, 1996.

\end{thebibliography}
\end{document}